\documentclass[a4paper,12pt]{article}

\usepackage[pdftex]{graphicx}
\usepackage[section] {placeins}
\usepackage{enumerate}
\usepackage[normalem]{ulem}
\usepackage{amsmath}
\usepackage{amsthm}
\usepackage{latexsym}
\usepackage{amsfonts}
\usepackage{amssymb}

\usepackage{tikz-cd}
\usetikzlibrary{graphs}
\usetikzlibrary{arrows.meta}
\usetikzlibrary{shapes.geometric}

\usepackage{authblk}

\usepackage{latexsym,amssymb,amsfonts, amsthm, amsmath}
\usepackage[english]{babel}
\usepackage{bm}
\usepackage{mathrsfs}

\newtheorem{thm}{Theorem}[section]
\newtheorem{exa}[thm]{Example}
\newtheorem{lem}[thm]{Lemma}
\newtheorem{cor}[thm]{Corollary}

\newtheorem{conj}[thm]{Conjecture}

\theoremstyle{definition}

\newcommand{\Z}{\mathbb{Z}}

\newcommand{\supp}{\mathop{\rm supp}\nolimits}

\newcommand{\apph}[3]{\mathrel{\operatorname*{{\uplus}^{#1}}_{#2}^{#3}}}

\usepackage{fullpage}

\begin{document}

\title{Grid-Based Graphs, Linear Realizations and the Buratti-Horak-Rosa Conjecture}

\author[1]{Onur A\u{g}{\i}rseven}  
\author[2]{M.~A.~Ollis\footnote{Corresponding author.  Email: \texttt{matt\_ollis@emerson.edu}} } 

\affil[1]{Unaffiliated}
\affil[2]{Marlboro Institute for Liberal Arts and Interdisciplinary Studies, Emerson College, Boston, Massachusetts 02116, USA}


\maketitle

\begin{abstract}
Label the vertices of the complete graph~$K_v$ with the integers~$\{0, 1, \ldots, v-1\}$ and define the {\em length}~$\ell$ of the edge between distinct vertices labeled~$x$ and~$y$ by~$\ell(x,y) = \min( |y-x|, v - |y-x| )$.  A {\em realization} of a multiset~$L$ of size~$v-1$ is a Hamiltonian path through~$K_v$ whose edge labels are~$L$.    The {\em Buratti-Horak-Rosa (BHR) Conjecture} is that there is a realization for a multiset~$L$ if and only if for any divisor~$d$ of~$v$ the number of multiples of~$d$ in~$L$ is at most~$v-d$.  

We introduce ``grid-based graphs" as a useful tool for constructing particular types of realizations, called ``linear realizations," especially when the multiset in question has a support of size~3.  This lets us prove many new instances of the BHR Conjecture, including those for multisets of the form~$\{1^a, x^b, y^c \}$ when~$a \geq x+y - \epsilon$, where~$\epsilon$ is the number of even elements in~$\{ x,y \}$, and those for all multisets of the following forms for sufficiently large~$v$ with $\gcd(v,y) = 1$ for all~$y \in L$:
\begin{itemize}
\item  $\{1^a, 2^b, x^c\}$, except possibly when~$a \in \{1,2\}$ and~$x$ is odd,
\item  $\{1^a, x^b, (x+1)^c\}$.
\end{itemize}
This establishes that there are infinitely many sets~$U$ of size~3 for which there are infinitely many values of~$v$ where the BHR Conjecture holds for each multiset with support~$U$.  We also show that the BHR~Conjecture holds for~$\{1^a,x^b,(x+1)^c\}$ when $x \in \{7,9,10\}$ and~$\gcd(v,x) = \gcd(v,x+1) = 1$.

\bigskip

\end{abstract}

\section{Introduction}

Consider the complete graph~$K_v$ on~$v$ vertices with labels~$\{0,1,\ldots, v-1 \}$.  Define an induced edge-labeling~$\ell$, called the~{\em (cyclic) length}, by setting
$$\ell(x,y) = \min( |y-x|, v - |y-x| )$$
for distinct vertices~$x$ and~$y$.  Call an edge with length~$\ell$ an {\em $\ell$-edge}.
The length is an integer in the range~$0 < \ell(x,y) \leq \lfloor v/2 \rfloor$.  If we place the vertices of~$K_v$ at evenly spaced points on a circle in the natural order then the length gives the shortest number of steps around the circle from~$x$ to~$y$.  In other words, $\ell(x,y) = i$, for $i$ in the range~$0 < i \leq  \lfloor v/2 \rfloor$, if and only if $y-x \equiv \pm i \pmod{v}$.

Let~$\Gamma$ be a Hamiltonian path of~$K_v$.  If~$L$ is the multiset of lengths of the edges of $\Gamma$ then $\Gamma$ {\em (cyclically) realizes}~$L$ and~$\Gamma$  is a {\em (cyclic) realization} of~$L$.   For example, if~$v = 11$ then the Hamiltonian path
$[ 0, 8,7,2,4,10,9,1,3,5,6 ]$
in~$K_v$  realizes the multiset~$\{ 1^3, 2^3,  3^2,  5^2 \}$.    (We use exponents in multisets to indicate multiplicity.  For a multiset~$L$, the set~$\supp(L) = \{ x : x \in L \}$ is the {\em support} of~$L$.)  

The Buratti-Horak-Rosa (BHR) Conjecture---posed by Buratti for prime~$v$~\cite{West07}, generalized to arbitrary~$v$ by Horak and Rosa~\cite{HR09} and given the following succinct formulation by Pasotti and Pellegrini~\cite{PP14b}---gives a condition that determines exactly when a multiset has a realization.  

\begin{conj}{\rm \cite{HR09,PP14b,West07} (BHR Conjecture)} 
Let~$L$ be a multiset of size~$v-1$ with $\supp(L) \subseteq \{ 1, \ldots, \lfloor v/2 \rfloor \}$.  There is a realization for~$L$ if and only if for any divisor~$d$ of~$v$ the number of multiples of~$d$ in~$L$ is at most~$v-d$.  
\end{conj}

The condition in the BHR Conjecture is necessary~\cite{HR09} and is vacuous when~$v$ is prime.  We say that the BHR~Conjecture holds for a particular support if it holds for all multisets with that support.

The BHR Conjecture is known to hold for multisets with support of size at most~2~\cite{DJ09,HR09}.
In this paper we construct realizations for multisets with small support, primarily those with support of size at most~3.   These results constitute a significant expansion of the list instances where the BHR Conjecture is known to hold.  In particular they imply that there are infinitely many sets~$U$ of size~3 for which there are infinitely many values of~$v$ where the Buratti-Horak-Rosa Conjecture holds for multisets with support~$U$; a step up in generality from previous results.

Theorem~\ref{th:known} collects the known results for multisets with support of size at most~3. 

\begin{thm}\label{th:known}{\rm (Known results)} 
Let $L$ be a multiset of size~$v-1$ having underlying set~$U$ with $|U| \leq 3$.  In each of the following cases, if $L$ is admissible, then it is realizable.
\begin{enumerate}
\item $|U| \leq 2$ \cite{DJ09,HR09},
\item 
  $\max(U) \leq 7$ or $U = \{ 1,2,8 \}, \{1,2,10\}, \{1,2,12 \}$  \cite{CD10,CO,PP14, PP14b}, 

\item $L = \{ 1^a, 4^4, 8^c\}$ when $a \geq 3$ \cite{Avila22}, 
\item $L = \{1^a,2^b, x^c  \}$ when $x$ is even and $a+b \geq x-1$ \cite{PP14},
\item $L = \{1^a, 2^{x-1}, x \}$ \cite{Avila22},
\item $L = \{1^a, 3^b, x^c \}$ when~$x$ is even, $c$ is odd, and either~$a \geq x+1$ or $a=x$ and $3 \nmid b$~\cite{Avila23}.
\item $L = \{1^a, x^b, (x+1)^c \}$ when $x$ is odd and either $a \geq \min( 3x-3, b+2x-3)$ or $a \geq 2x-2$ and  $c \geq 4b/3$ \cite{OPPS},
\item $L = \{1^a, x^b, (x+1)^c \}$ when $x$ is even and either $a \geq \min( 3x-1, c+2x-1)$ or $a \geq 2x-1$ and  $b \geq c$ \cite{OPPS},
\item $L = \{ 1^a, x^b, (2x)^c \}$ when $a \geq x-2$, $c$ is even and $b \geq 5x-2+c/2$ \cite{OPPS2},
\item $L = \{ 1^a, x^b, y^c \}$ when $x$ is even,  $x<y$ and either $y$ is even and $a \geq y-1$ or $y$ is odd and $a \geq 3y-4$  \cite{OPPS},
\item $L = \{ 1^a, x^b , y^c \}$ where $x < y$ and $a \geq x + 4y - 5$ \cite{HR09,OPPS2},
\item $v \leq 37$ \cite{MP,Meszka}.
\end{enumerate}
\end{thm}

That most of these results have~1 as an element of the multiset may seem limiting.  However, if we consider the vertex labels to be elements of the cyclic group~$\Z_v$ then we can use automorphisms to turn the problem of realizing one multiset into another.  In particular, given~$x$ with $1 \leq x < v$, let the {\em reduced form of~$x$ with respect to~$v$} be~$\widehat{x} = \min(x, v-x)$.  Let~$s$ be coprime to~$v$.  Then multiplication by~$s$ in~$\Z_v$ is an automorphism and a realization of a multiset~$\{1^a, x^b, y^c \}$ is equivalent to one for~$\{\widehat{s}^a, \widehat{sx}^b, \widehat{sy}^c \}$.  (A similar argument applies to multisets with larger support.)  It follows that if one of the elements of a multiset is coprime to~$v$ (which is guaranteed to be the case when~$v$ is prime) then the question of whether that multiset is realizable is equivalent to a question for a particular multiset that has~1 as an element~\cite{HR09}.  

Almost all of the new results of this paper are for multisets that have~1 as an element.  In Section~\ref{sec:bhr} we show how the new and existing constructions for realizations give strong results for the BHR Conjecture.

Suppose we label the edge between each pair of vertices~$x$ and~$y$ of~$K_v$ with $|y-x|$ instead of $\ell(x,y)$.   This corresponds to the distance between vertices when placed at evenly spaced points on a line in consecutive order; call it the {\em linear length}.  A Hamiltonian path~$\Gamma$ again gives rise to a multiset~$L$ of the labels of its edges; in this case~$\Gamma$ is a {\em linear realization} of~$L$.   Typically, a Hamiltonian path will linearly realize a different multiset to the one that it cyclically realizes.  For example, the path $[0,8,7,2,4,10,9,1,3,5,6]$ in~$K_{11}$ that we saw cyclically realizes~$\{1^3, 2^3, 3^2, 5^2\}$ is a linear realization of~$\{1^3, 2^3, 5, 6,  8^2\}$.

If the largest linear length in a linear realization is at most $\lfloor v/2 \rfloor$ then a  linear realization is a cyclic realization of the same multiset.  As we are interested almost entirely in this situation, we refer to linear lengths simply as ``lengths" and disambiguate only when necessary.  Items~3--8 and~10--11 of Theorem~\ref{th:known} were proved by constructing linear realizations that are also cyclic realizations for the same multiset.

In this paper we develop the theory of linear realizations.  This builds on and unifies some of the existing ideas in the literature for constructing linear realizations.  Theorem~\ref{th:main} summarizes the main new results.  In many instances the constructions used give improvements for a variety of subcases of these results and related constructions prove less strong, but still new, results for similar multisets. 

\begin{thm}\label{th:main}{\rm (Main Results)}
Let $L = \{1^a, x^b, y^c \}$ with $1< x<y$ and $a+b \geq y-1$.  Then~$L$ has a linear realization in the following cases:
\begin{enumerate}
\item $x$ even and $a \geq x+y-2$,
\item $x$ odd and $a \geq x+y -1$ if~$y$ is even or   $a \geq x+y $ if~$y$ is odd,
\item $y = x+1$ and $a \geq x+1$ 
\item $x$ even, $y \geq 2x+1$, $a \geq 3x-2$ and $a+b \geq x+y-1$,
\item $x=2$, $y\geq5$ and $a \geq 3$. 
\end{enumerate}
\end{thm}

Section~\ref{sec:dtt} introduces some foundational  concepts and tools  that are used throughout the paper.   Section~\ref{sec:supp2}  considers linear realizations for multisets with supports of size~2; these become useful in constructing linear realizations for multisets with support of size~3 in the subsequent sections.   Section~\ref{sec:1xx+1} considers cases with~$y=x+1$ and Section~\ref{sec:evenx} considers cases with~$x$ even (where~$x<y$).  Section~\ref{sec:bhr}  uses automorphisms of~$\Z_v$ to prove new results about the BHR Conjecture from linear realizations, including the following result:

\begin{thm}\label{th:bhr}{\rm (Implications for BHR)}
For a given~$x$, the BHR Conjecture holds for all sufficiently large~$v$ with $\gcd(v,x) = \gcd(v,x+1) = 1$ for supports of the form~$\{1, 2, x \}$ and $\{1,x,x+1\}$, with the possible exception of~$L = \{1^a, 2^b, x^c\}$ with~$x$ odd and $a<3$.
\end{thm}

Note that if~$v$ is prime (and so we are considering Buratti's original conjecture) then~$v$ automatically satisfies the~$\gcd$ conditions of Theorem~\ref{th:bhr}.    This result is the first to demonstrate that the BHR~Conjecture holds for each of infinitely many supports of size~3 at infinitely many orders.

In Section~6 we also demonstrate that when~$x \leq 15$ that the BHR Conjecture holds  for~$\{1,x,x+1\}$ for all~$v$ with $\gcd(v,x) = \gcd(v,x+1) = 1$ with at most~15 remaining pairs~$(x,v)$.  There are no exceptions with~$x  \leq 7$ or $x \in \{ 9,10\}$ and in~\cite{AO+} new constructions deal with the remaining~15 cases.

\section{Definitions, tools and techniques}\label{sec:dtt}

We shall think about linear realizations in a geometric way, using the integer lattice.  For some~$x>1$, label the vertices of the lattice so that a vertex labeled~$a$ has~$a+1$ immediately to the right of it and~$a+x$ immediately above it.
We take~$v$ of the vertices, labeled $0, 1, \ldots v-1$, to be the vertices of~$K_v$, typically chosen to form an approximately rectangular region of width~$x$.  Horizontally adjacent vertices are connected by an edge with length~1 and vertically adjacent vertices are connected by an edge of length~$x$, facilitating visualization of realizations of multisets with support~$\{1,x\}$.  

A common situation is to take $v = qx + r$, with $0 \leq r < x$, and to use a rectangular grid of width~$x$ and height~$q$, with~$r$ further vertices forming an additional partial row above and/or below the rectangle.
When the support is~$\{1,x\}$, we construct realizations that use only horizontally and vertically adjacent edges and so are subgraphs of the lattice.

The first diagram in Figure~\ref{fig:intro_eg} shows the linear realization
$$[ 0,7,14,21,22,15,8,1,2,9,16,23,24,17,10,3,4,11,18,19,12,5,6,13,20]$$
for the multiset~$\{1^6,  7^{18}\}$.

\begin{figure}[tp]
\caption{Linear realizations for~$\{1^6,7^{18}\}$ and~$\{ 1^7, 7^{17}, 8^3\}$. }\label{fig:intro_eg}
\begin{center}
\begin{tikzpicture}[scale=0.9, every node/.style={transform shape}]

\fill (0,4) circle (2pt) ;
\fill (0,1) circle (2pt) ;
\fill (0,2) circle (2pt) ;
\fill (0,3) circle (2pt) ;

\fill (1,4) circle (2pt) ;
\fill (1,1) circle (2pt) ;
\fill (1,2) circle (2pt) ;
\fill (1,3) circle (2pt) ;

\fill (2,4) circle (2pt) ;
\fill (2,1) circle (2pt) ;
\fill (2,2) circle (2pt) ;
\fill (2,3) circle (2pt) ;

\fill (3,4) circle (2pt) ;
\fill (3,1) circle (2pt) ;
\fill (3,2) circle (2pt) ;
\fill (3,3) circle (2pt) ;

\fill (4,3) circle (2pt) ;
\fill (4,1) circle (2pt) ;
\fill (4,2) circle (2pt) ;

\fill (5,3) circle (2pt) ;
\fill (5,1) circle (2pt) ;
\fill (5,2) circle (2pt) ;

\fill (6,3) circle (2pt) ;
\fill (6,1) circle (2pt) ;
\fill (6,2) circle (2pt) ;

\draw (0,1) -- (0,4) -- (1,4) -- (1,1) -- (2,1) -- (2,4) -- (3,4) -- (3,1) 
   -- (4,1) -- (4,3) -- (5,3) -- (5,1) -- (6,1) --  (6,3) ;

\node at (-0.3, 1.2) {\tiny 0} ;
\node at (-0.3, 2.2) {\tiny 7} ;
\node at (-0.3, 3.2) {\tiny 14} ;
\node at (-0.3, 4.2) {\tiny 21} ;

\node at (0.7, 1.2) {\tiny 1} ;
\node at (0.7, 2.2) {\tiny 8} ;
\node at (0.7, 3.2) {\tiny 15} ;
\node at (0.7, 4.2) {\tiny 22} ;

\node at (1.7, 1.2) {\tiny 2} ;
\node at (1.7, 2.2) {\tiny 9} ;
\node at (1.7, 3.2) {\tiny 16} ;
\node at (1.7, 4.2) {\tiny 23} ;

\node at (2.7, 1.2) {\tiny 3} ;
\node at (2.7, 2.2) {\tiny 10} ;
\node at (2.7, 3.2) {\tiny 17} ;
\node at (2.7, 4.2) {\tiny 24} ;

\node at (3.7, 1.2) {\tiny 4} ;
\node at (3.7, 2.2) {\tiny 11} ;
\node at (3.7, 3.2) {\tiny 18} ;

\node at (4.7, 1.2) {\tiny 5} ;
\node at (4.7, 2.2) {\tiny 12} ;
\node at (4.7, 3.2) {\tiny 19} ;

\node at (5.7, 1.2) {\tiny 6} ;
\node at (5.7, 2.2) {\tiny 13} ;
\node at (5.7, 3.2) {\tiny 20} ;

\fill (8,1) circle (2pt) ;
\fill (8,2) circle (2pt) ;
\fill (8,3) circle (2pt) ;
\fill (8,4) circle (2pt) ;

\fill (9,1) circle (2pt) ;
\fill (9,2) circle (2pt) ;
\fill (9,3) circle (2pt) ;
\fill (9,4) circle (2pt) ;

\fill (10,1) circle (2pt) ;
\fill (10,2) circle (2pt) ;
\fill (10,3) circle (2pt) ;
\fill (10,4) circle (2pt) ;

\fill (11,1) circle (2pt) ;
\fill (11,2) circle (2pt) ;
\fill (11,3) circle (2pt) ;
\fill (11,4) circle (2pt) ;

\fill (12,1) circle (2pt) ;
\fill (12,2) circle (2pt) ;
\fill (12,3) circle (2pt) ;
\fill (12,4) circle (2pt) ;

\fill (13,1) circle (2pt) ;
\fill (13,2) circle (2pt) ;
\fill (13,3) circle (2pt) ;
\fill (13,4) circle (2pt) ;

\fill (14,0) circle (2pt) ;
\fill (14,1) circle (2pt) ;
\fill (14,2) circle (2pt) ;
\fill (14,3) circle (2pt) ;

\draw (14,0) -- (14,3) -- (13,3) -- (13,1) -- (12,1) -- (12,3) -- (13,4) -- (11,4) -- (11,1) -- (9,1)
   -- (10,2) -- (10,4) -- (9,4) -- (9,2) -- (8,1) -- (8,4) ;

\node at (13.7, 0.2) {\tiny 0} ;
\node at (13.7, 1.2) {\tiny 7} ;
\node at (13.7, 2.2) {\tiny 14} ;
\node at (13.7, 3.2) {\tiny 21} ;

\node at (7.7, 1.2) {\tiny 1} ;
\node at (7.7, 2.2) {\tiny 8} ;
\node at (7.7, 3.2) {\tiny 15} ;
\node at (7.7, 4.2) {\tiny 22} ;

\node at (8.7, 1.2) {\tiny 2} ;
\node at (8.7, 2.2) {\tiny 9} ;
\node at (8.7, 3.2) {\tiny 16} ;
\node at (8.7, 4.2) {\tiny 23} ;

\node at (9.7, 1.2) {\tiny 3} ;
\node at (9.7, 2.2) {\tiny 10} ;
\node at (9.7, 3.2) {\tiny 17} ;
\node at (9.7, 4.2) {\tiny 24} ;

\node at (10.7, 1.2) {\tiny 4} ;
\node at (10.7, 2.2) {\tiny 11} ;
\node at (10.7, 3.2) {\tiny 18} ;
\node at (10.7, 4.2) {\tiny 25} ;

\node at (11.7, 1.2) {\tiny 5} ;
\node at (11.7, 2.2) {\tiny 12} ;
\node at (11.7, 3.2) {\tiny 19} ;
\node at (11.7, 4.2) {\tiny 26} ;

\node at (12.7, 1.2) {\tiny 6} ;
\node at (12.7, 2.2) {\tiny 13} ;
\node at (12.7, 3.2) {\tiny 20} ;
\node at (12.7, 4.2) {\tiny 27} ;

\end{tikzpicture}\end{center}
\end{figure}
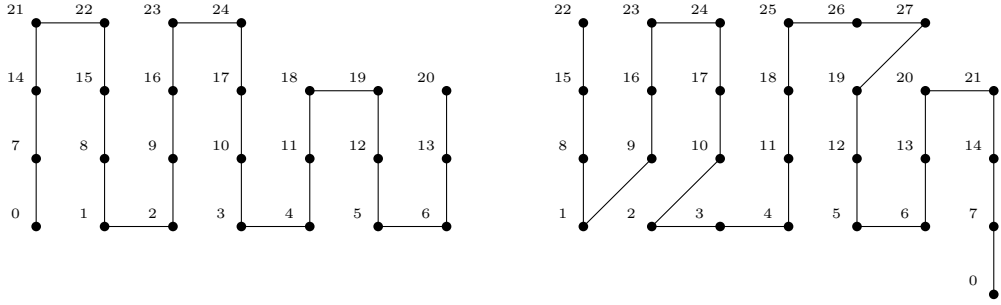

When the support is of the form~$\{1,x,y\}$ there will be one further type of edge to consider.  For example, if $y= x+1$ then we use edges that use diagonally adjacent vertices in the northeast/southwest direction.
The second diagram in Figure~\ref{fig:intro_eg} shows  the linear realization
$$[ 0,7,14,21,20,13,6,5,12,19,27,26,25,18,11,4,3,2,10,17,24,23,16,9,1,8,15,22 ]$$
for the multiset~$\{ 1^7, 7^{17}, 8^3\}$.   In this figure all vertices are labeled.  In future figures, to make the geometric patterns clearer, typically only a few vertices near the sides are labeled.

As well as~$y=x+1$, in this paper we use third elements corresponding to two other types of edge.  When~$y = x-1$ we have the edge between diagonally adjacent vertices in the northwest/southeast direction and when~$y$ is smaller than~$x-1$ we have an edge between vertices in the same row at horizontal distance~$y$.   In our constructions, all of these $y$-edges are incident with at least one vertex on a side of the grid.

We refer to graphs whose vertices are on the integer lattice chosen in such a way that vertical edges correspond to an element~$x$ as {\em grid-based} graphs.

The necessary condition in the BHR Conjecture follows from consideration of cosets of a subgroup in~$\Z_v$.  Here we prove Lemma~\ref{lem:fauxsets}, a parallel result for linear realizations.  Given~$x$ in the range~$1 <  x \leq \lfloor v/2 \rfloor$ and some~$k$ in the range~$0 \leq k < x-1$, define the {\em fauxset} $\varphi_k$ to be the vertex labels of~$K_v$ that are congruent to~$k \pmod{x}$.  Visually, in a grid-based graph with width~$x$, each set of edges in a vertical line segment is a fauxset.   When~$v$ is a multiple of~$x$ (and the grid-based graph is a rectangle) a fauxset is the same thing as a coset in~$\Z_v$,

\begin{lem}\label{lem:fauxsets}{\rm (Connecting fauxsets)}
Let~$L = \{ x_1^{a_1}, x_2^{a_2}, \ldots, x_t^{a_t} \}$ with $a_t > 0$.  If there is a linear realization for~$L$, then necessarily  $\sum_{i=1}^{t-1} a_i \geq x_t - 1$. 
\end{lem}

\begin{proof}
Consider the fauxsets with respect to~$x_t$; there are $x_t$ of them.  A Hamiltonian path must visit each of these fauxsets and, unlike in a cyclic realization, only edges of lengths other than~$x_t$ move from one fauxset to another.  Hence there must be at least~$x_t - 1$ edges of length other than~$x_t$.
\end{proof}

We often consider multisets of the form~$\{1^a, x^b, y^c\}$ with $1 < x < y$.  In this situation Lemma~\ref{lem:fauxsets} says that for a linear realization to exist we must have $a+b \geq y-1$, hence the appearance of this condition in the statement of Theorem~\ref{th:main}.   Lemma~\ref{lem:fauxsets} also requires linear realizations to have~$a+c \geq x-1$ in this situation; in our results this usually follows immediately from constraints on~$a$ and/or~$x$.

In distinction to the parallel result for arbitrary realizations, the condition of Lemma~\ref{lem:fauxsets} is definitely not sufficient for the existence of a linear realization.   For example, in Section~\ref{sec:supp2} we see instances of multisets with support~$\{1, x\}$ that meet the condition but have no linear realization.

A linear realization in $K_v$ is {\em standard} if an end-vertex label is~0.  It is {\em perfect} if the two end-vertex labels are~$0$ and~$v-1$.    The realizations in Figure~\ref{fig:intro_eg} are standard but not perfect.

Given a Hamiltonian path in~$K_v$ we can form another one by replacing each vertex label~$x$ with~$v-1-x$.  This is the {\em complement} of the original path.  The complement realizes the same multiset as the original path.  If the original path is standard then the complement has~$v-1$ as an end-vertex label; if the original path is perfect then so is the complement.  

Given a Hamiltonian path in~$K_v$ we can form a non-Hamiltonian path in~$K_{v+t}$ by adding~$t$ to each vertex label.  This is the {\em (embedded) translation} of the original path by~$t$.  It realizes the same multiset of the original path.  If the original realization is standard then this translation has~$t$ as an end-vertex label; if the original realization is perfect then this translation has~$t$ and~$v+t$ as the end-vertex labels.

\begin{lem}\label{lem:concat}{\rm \cite{HR09,OPPS} (Concatenation) }
Suppose~$L$ and~$M$ are multisets that have a standard linear realization.   Then~$L \cup M$ has a linear realization.  If the standard linear realization of~$M$ is perfect then $L \cup M$ has a standard linear realization.  If both standard realizations are perfect then~$L \cup M$ has a perfect realization.
\end{lem}

\begin{proof}
Suppose~$|L| = v$ and~$|M|=w$ and let~${\mathbf g}$ and ${\mathbf h}$ be standard linear realizations of~$L$ and~$M$ respectively.    Define a Hamiltonan path in~$K_{v+w+1}$, called the {\em concatenation} of ${\mathbf g}$ and~${\mathbf h}$ and denoted ${\mathbf g} \oplus {\mathbf h}$, by taking the complement of~${\mathbf g}$  and the translation of~${\mathbf h}$ by~$v-1$ and identifying the end-vertices labeled with~$v-1$.

The properties claimed follow immediately from the discussion prior to the statement of the lemma.
\end{proof}

We can interpret Lemma~\ref{lem:concat} in terms of the geometry of grid-based graphs.  Taking the complement is equivalent to rotating the graph through~$180^\circ$; taking a translation is equivalent to adding more points that get labeled with lower labels.   
The following consequence of Lemma~\ref{lem:concat} allows us to reach realizations for multisets where the number of 1-edges is arbitrary, subject to some minimum value.

\begin{lem}{\rm \cite{OPPS} (Appending 1-edges)}\label{lem:ones}
If~$L$ has a standard (respectively perfect) realization then $L \cup \{1^s\}$ has a standard (respectively perfect) realization for all~$s$.
\end{lem}

\begin{proof}
For the multiset~$\{1^s\}$, the Hamiltonian path $[0,1, \ldots, s-1]$ is a perfect realization.  Using this in Lemma~\ref{lem:concat} we see that if~$L$ has a standard (respectively perfect) realization then $L \cup \{1^s \}$ has a standard (respectively perfect) realization for all~$s$, as required.
\end{proof}

We shall also need another way to combine linear realizations that have particular properties.

\begin{lem}{\rm \cite{HR09} (Insertion between~0 and~1) }\label{lem:hrjoin}
Suppose that $L$ has a realization in which~$0$ and~$1$ are adjacent vertices and that~$M$ has a standard linear realization with last vertex label~$1$.  Then~$(L \cup M) \setminus \{1\}$ has a linear realization.
\end{lem}

\begin{proof}
Suppose~$|L| = v$ and~$|M|=w$ and let~${\mathbf g}$ and ${\mathbf h}$ be linear realizations of~$L$ and~$M$ respectively with the properties of the theorem statement.     We define a Hamiltonan path in~$K_{v+w}$ 
that realizes~$(L \cup M) \setminus \{1\}$.  Take the complement of~${\mathbf g}$, which has an edge between~$v-1$ and~$v$. Replace this edge with the translation of~${\mathbf h}$ by~$v-1$, which has end-vertices~$v-1$ and~$v$. This gives the required path.
\end{proof}

We introduce some notation that allows us to more concisely express the constructions of standard linear realizations.  It takes advantage of the fact that most of the constructions consist of a large number of in-fauxset edges and a small number of between-fauxset ones.  Again, let~$v = qx+r$ and consider the fauxsets~$\varphi_k$ for $0 \leq k < x$.  Let
$$q^* = q^*(v,r,k) =  
\begin{cases} 
q & \mathrm{if \  } k<r, \\
 q-1 & \mathrm{otherwise,}
\end{cases}
$$
so $q^*x + k$ is the largest vertex label of~$\varphi_k$.
 Let $\Psi_k^{(a,b)}$ denote the path 
$[ ax + k, (a+1)x+k, \ldots, bx+k ]$ of consecutive labels in~$\varphi_k$. If~$a=b$ (so the path is a single vertex) then we do not repeat the value and write $\Psi_k^{(a)}$. If $a=0$ and $b=q^*$ (so the path covers the whole fauxset) then we omit the superscript and write~$\Psi_k$.  

Given two paths~$\mathbf{p_1}$ and~$\mathbf{p_2}$, we call the path obtained by adding an $\ell$-edge between them a {\em bridge concatenation}, denoted by $\mathbf{p_1} \uplus^\ell \mathbf{p_2}$.  Note that~$\mathbf{p_1}$ and~$\mathbf{p_2}$ may be traversed in either direction and there might be more than one pair of ends of the respective paths with difference~$\ell$, so a bridge concatenation is not necessarily unique in general.  However, the choice to be made will always be clear in the constructions.  In particular, Hamiltonian paths described here usually start at~0 (that is, the realizations are standard), which indicates the direction of travel through the first part of the first fauxset~$\varphi_0$.    Frequently we take~$\ell=1$; in this case we omit the superscript and write~$\mathbf{p_1} \uplus \mathbf{p_2}$.  We denote repeated~$\uplus$ operations performed on $\mathbf{p_1}, \mathbf{p_2}, \ldots, \mathbf{p_m}$ in turn  in the obvious way as 
$$ \apph{{\it \ell}}{k=1}{m}  \mathbf{p_k} {\rm \ \ \ \ \ \ or \ \ \ \ \ \ \ }  \apph{}{k=1}{m} \mathbf{p_k}$$
and denote the operations on the same paths in the reverse order as
$$ \apph{{\it \ell}}{k=m}{1}  \mathbf{p_k} {\rm \ \ \ \ \ \ or \ \ \ \ \ \ \ }  \apph{}{k=m}{1} \mathbf{p_k}.$$
With this notation, the paths of Figure~\ref{fig:intro_eg} are
$$ \apph{}{k=0}{6} \Psi_k  {\rm \ \  \ and \  \ \ } 
  \Psi_0 \ \uplus \  \Psi_6^{(0,2)} \ \uplus \  \Psi_5^{(0,2)} \ \uplus^8 \  \Psi_6^3 \ \uplus  \ \Psi_5^3 \ \uplus \ \Psi_4 \ \uplus \  \Psi_3^0 \ \uplus \  \Psi_2^0 \ \uplus^8 \ 
    \Psi_3^{(1,3)} \ \uplus \  \Psi_2^{(1,3)} \ \uplus^8 \ \Psi_1  $$
respectively (when we consider paths similar to the second one, in Section~\ref{sec:1xx+1},  we shall have an alternative notational approach available).

Finally for this section we give a new construction that produces realizations that can be used in Lemma~\ref{lem:hrjoin} and illustrates this new notation.  We shall use it in Section~\ref{sec:evenx} when constructing linear realizations for multisets with support~$\{1,x,y\}$ where~$x<y$ and~$x$ is even.

\begin{lem}\label{lem:2rect}{\rm (Bridging with a partial row)} 
Let~$L_1$ and~$L_2$ be multisets of sizes~$v_1-1$ and~$v_2 - 1$ respectively and let~$x = v_1 + v_2$.  Let~$b$ be such that $b \equiv v_1 +1 \pmod{x}$.  If~$L_1$ and~$L_2$ have standard linear realizations, then there is a linear realization for~$L = L_1 \cup L_2 \cup \{1, x^b \}$ that has an edge between~$0$ and~$1$.
\end{lem}

\begin{proof}
Let~$\mathbf{g} = [g_1, g_2, \ldots, g_{v_1}]$ and $\mathbf{h} = [h_1, h_2, \ldots, h_{v_2}]$, where~$g_1 = 0 = h_1$,  be standard linear realizations of~$L_1$ and~$L_2$ respectively.  

Observe that~$|L_1| + |L_2| + 1 = x-1$ and so there is the minimum number of ``hops" between the fauxsets with respect to~$x$ required by Lemma~\ref{lem:fauxsets}.  The desired linear realization must therefore traverse each fauxset fully as a sequence of consecutive vertex labels.

The sequence $[x-h_{v_2}, x-h_{v_2 - 1}, \ldots, x-h_1]$ is a translation of the complement of~$\mathbf{h}$, that uses the labels~$v_1, v_1+1, \ldots,  x$. Let the sequence of edge lengths be $\mathbf{e}= [e_1, e_2, \ldots, e_{v_2 -1}]$, which
is the  edge length sequence of~$\mathbf{h}$ in reverse order.  The sequence $[g_1 + 1, g_2+1, \ldots, g_{v_1}+1]$
is a translation of~$\mathbf{g}$ that uses the labels~$1,2, \ldots, v_1$.  Let the sequence of edge lengths be $\mathbf{d} = [d_1, d_2, \ldots, d_{v_1-1}]$, which is the same as the  edge length sequence of~$\mathbf{g}$.

Let~$v = |L| +1$ and
consider the following path:
$$\left( \apph{\mathbf{e}}{i= x-h_{v_2}}{x-h_1} \Psi_i \right)  \uplus \left(  \apph{\mathbf{d}}{i=g_1+1}{g_{v_1}+1} \Psi_i    \right)$$
where~$\mathbf{d}$ indicates that those fauxsets are joined with the edge length sequences of~$\mathbf{d}$ in turn, and similarly for~$\mathbf{e}$, and 
where we choose the starting point so that the hop to~$\Psi_x = \Psi_0$, which is the last step in the first parentheses, occurs between the largest vertex labels of~$\varphi_{x-h_2}$ and~$\varphi_0$.

As~$b \equiv v_1 +1 \pmod{x}$, the fauxsets~$\varphi_0, \ldots, \varphi_{v_1 -1}$ are all the same size, as are the fauxsets~$\varphi_{v_1}, \ldots, \varphi_{x-1}$.  
The hops between fauxsets are valid as all of the fauxsets within each pair of parentheses are the same size (with the exception of~$\Psi_0$ in the first pair, which we have dealt with in choosing the starting point of the path).  All fauxsets are covered and so  this is a Hamiltonian path.  Between fauxsets it has edge lengths~$L_2$ from the first pair of parentheses, $L_1$ from the second pair of parentheses, and~$1$ where the two are joined.  This 1-edge is between vertices~0 and~1.  Within the fauxsets we get~$b$ edges of length~$x$.   Hence we have the required linear realization.
\end{proof}

Figure~\ref{fig:129} illustrates the construction of Lemma~\ref{lem:2rect} for~$L_1 = \{1, 2^3 \}$, which has standard linear realization~$(0,2,4,3,1)$, and $L_2 = \{ 1, 2^2 \}$, which has standard linear realization~$(0,2,3,1)$.  We have~$x=5+4 =9$ and take~$b = 24$, which is congruent to~$5+1 \pmod{9}$.

\begin{figure}[tp]
\caption{Linear realizations of~$\{1^{3}, 2^5, 9^{24}\}$ and $\{1^2, 2^5, 8^{21} \}$, constructed via Lemmas~\ref{lem:2rect} and~\ref{lem:2squash} respectively.}\label{fig:129}
\begin{center}
\vspace{5mm}
\begin{tikzpicture}[scale=0.86, every node/.style={transform shape}]

\fill (1,2) circle (2pt) ;
\fill (1,3) circle (2pt) ;
\fill (1,4) circle (2pt) ;

\fill (2,2) circle (2pt) ;
\fill (2,3) circle (2pt) ;
\fill (2,4) circle (2pt) ;

\fill (3,2) circle (2pt) ;
\fill (3,3) circle (2pt) ;
\fill (3,4) circle (2pt) ;

\fill (4,1) circle (2pt) ;
\fill (4,2) circle (2pt) ;
\fill (4,3) circle (2pt) ;
\fill (4,4) circle (2pt) ;

\fill (5,1) circle (2pt) ;
\fill (5,2) circle (2pt) ;
\fill (5,3) circle (2pt) ;
\fill (5,4) circle (2pt) ;

\fill (6,1) circle (2pt) ;
\fill (6,2) circle (2pt) ;
\fill (6,3) circle (2pt) ;
\fill (6,4) circle (2pt) ;

\fill (7,1) circle (2pt) ;
\fill (7,2) circle (2pt) ;
\fill (7,3) circle (2pt) ;
\fill (7,4) circle (2pt) ;

\fill (8,1) circle (2pt) ;
\fill (8,2) circle (2pt) ;
\fill (8,3) circle (2pt) ;
\fill (8,4) circle (2pt) ;

\fill (9,1) circle (2pt) ;
\fill (9,2) circle (2pt) ;
\fill (9,3) circle (2pt) ;
\fill (9,4) circle (2pt) ;

\draw (2,2) -- (1,2) -- (1,4)   ;
\draw (2,2) -- (2,4)  ;
\draw (3,2) -- (3,4)  ;
\draw (5,1) -- (4,1) -- (4,4)  ;
\draw (5,1) -- (5,4)  ;
\draw (6,1) -- (6,4)  ;
\draw (7,1) -- (7,4)  ;
\draw (8,1) -- (8,4)  ;
\draw (9,1) -- (9,4) -- (8,4)  ;

\draw  plot [smooth] coordinates {(1,4) (2,4.5) (3,4)};
\draw  plot [smooth] coordinates {(2,4) (3,4.5) (4,4)};
\draw  plot [smooth] coordinates {(5,4) (6,4.5) (7,4)};
\draw  plot [smooth] coordinates {(6,1) (7,0.5) (8,1)};
\draw  plot [smooth] coordinates {(7,1) (8,0.5) (9,1)};

\node at (3.7, 1) {\tiny 0} ;  
\node at (3.7, 2) {\tiny 9} ;  
\node at (2.7, 2) {\tiny 8} ;  
\node at (0.7, 2) {\tiny 6} ;  
\node at (0.7, 4) {\tiny 24} ;  
\node at (5.3, 1) {\tiny 1} ;  
\node at (6.3, 4) {\tiny 29} ;  
\node at (9.3, 4) {\tiny 32} ;  
\node at (9.3, 1) {\tiny 5} ;  

\fill (11,2) circle (2pt) ;
\fill (11,3) circle (2pt) ;
\fill (11,4) circle (2pt) ;

\fill (12,2) circle (2pt) ;
\fill (12,3) circle (2pt) ;
\fill (12,4) circle (2pt) ;

\fill (13,2) circle (2pt) ;
\fill (13,3) circle (2pt) ;
\fill (13,4) circle (2pt) ;

\fill (14,1) circle (2pt) ;
\fill (14,2) circle (2pt) ;
\fill (14,3) circle (2pt) ;
\fill (14,4) circle (2pt) ;

\fill (15,1) circle (2pt) ;
\fill (15,2) circle (2pt) ;
\fill (15,3) circle (2pt) ;
\fill (15,4) circle (2pt) ;

\fill (16,1) circle (2pt) ;
\fill (16,2) circle (2pt) ;
\fill (16,3) circle (2pt) ;
\fill (16,4) circle (2pt) ;

\fill (17,1) circle (2pt) ;
\fill (17,2) circle (2pt) ;
\fill (17,3) circle (2pt) ;
\fill (17,4) circle (2pt) ;

\fill (18,1) circle (2pt) ;
\fill (18,2) circle (2pt) ;
\fill (18,3) circle (2pt) ;
\fill (18,4) circle (2pt) ;

\draw (12,2) -- (11,2) -- (11,4)   ;
\draw (12,2) -- (12,4)  ;
\draw (13,2) -- (13,4)  ;
\draw (14,1)  -- (14,4)  ;
\draw (15,1) -- (15,4)  ;
\draw (16,1) -- (16,4)  ;
\draw (17,1) -- (17,4)  ;
\draw (18,4) -- (18,1) -- (17,1)  ;

\draw  plot [smooth] coordinates {(11,4) (12,4.5) (13,4)};
\draw  plot [smooth] coordinates {(12,4) (13,4.5) (14,4)};
\draw  plot [smooth] coordinates {(15,4) (16,4.5) (17,4)};
\draw  plot [smooth] coordinates {(14,1) (15,0.5) (16,1)};
\draw  plot [smooth] coordinates {(16,4) (17,4.5) (18,4)};

\node at (13.7, 1) {\tiny 0} ;  
\node at (13.7, 2) {\tiny 8} ;  
\node at (12.7, 2) {\tiny 7} ;  
\node at (10.7, 2) {\tiny 5} ;  
\node at (10.7, 4) {\tiny 21} ;  
\node at (15.3, 1) {\tiny 1} ;   
\node at (18.3, 4) {\tiny 28} ;  
\node at (18.3, 1) {\tiny 4} ;  

\end{tikzpicture}
\end{center}
\end{figure}
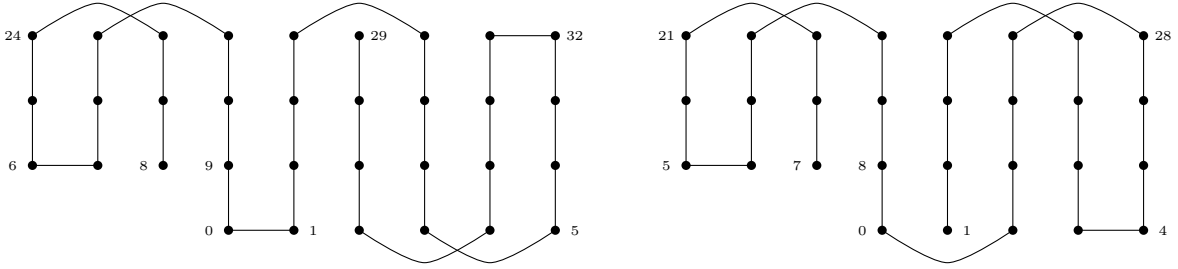

Although we don't use it in this paper, for completeness we record the following easy variation of the construction of Lemma~\ref{lem:2rect}.  It is obtained by removing the edge between~0 and~1 and having~$\Psi_0$ play the role of~$\Psi_1$.

\begin{lem}\label{lem:2squash}{\rm (Bridging with a partial row)} 
Let~$L_1$ and~$L_2$ be multisets of sizes~$v_1-1$ and~$v_2 - 1$ respectively and let~$x = v_1 + v_2 - 1$.  If~$L_1$ and~$L_2$ have standard linear realizations, then there is a linear realization for~$L = L_1 \cup L_2 \cup \{ x^b \}$ whenever~$b \equiv v_1 \pmod{x}$.
\end{lem} 

The second diagram in Figure~\ref{fig:129} illustrates this with the same~$L_1$ and~$L_2$ and their standard realizations as the first diagram, but now with~$x=8$ and~$b = 21$.

\section{Realizations with Support of Size 2}\label{sec:supp2}

In this section  we consider linear realizations for multisets~$L$ that have $|\supp(L)| = 2$.  First we look at realizations when the support has two consecutive elements in Lemma~\ref{lem:perf1}, which provides two families of perfect linear realizations.  
We then move on to consider the case $\supp(L) = \{1,x\}$, classifying exactly when a standard linear realization exists in Theorem~\ref{th:omega}.

\begin{lem}\label{lem:perf1}{\rm (Useful perfect realizations)} 
For any $x$, there is a perfect linear realization for $\{ (x-1)^{x-1} , x^{x} \}$ and for $\{ x^{x+2} , (x+1)^{x-1} \}$.
\end{lem}

\begin{proof}
First let~$v = 2x$ and consider the multiset  $\{ (x-1)^{x-1} , x^{x} \}$.
This is realized by the Hamiltonian path
$$\mathbf{g_1} = \apph{\mathit{x}-1}{k=0}{x-1} \Psi_{k} =  [0, x ;  1, x+1 ;  2, x+2 ; \ldots ; x-2, 2x-2; x-1, 2x-1 ]$$
(where we use semi-colons to indicate the bridges between fauxsets).

Now let~$v = 2x+2$ and consider the multiset  $\{ x^{x+2} , (x+1)^{x-1} \}$.  This is realized by the Hamiltonian path
$$\mathbf{g_2} = \Psi_0 \uplus^{\mathit{x}+1} \left( \apph{\mathit{x}+1}{k=x-1}{1} \Psi_k \right) = [ 0, x, 2x; x-1, 2x-1; x-2, 2x-2; \ldots ;  2, x+2; 1, x+1, 2x+1 ].$$

Figure~\ref{fig:xx+1} shows the paths as grid-based graphs with the $x$-edges vertical and the $(x \pm 1)$-edges diagonal, making it straightforward to check that they have the required properties.
\end{proof}

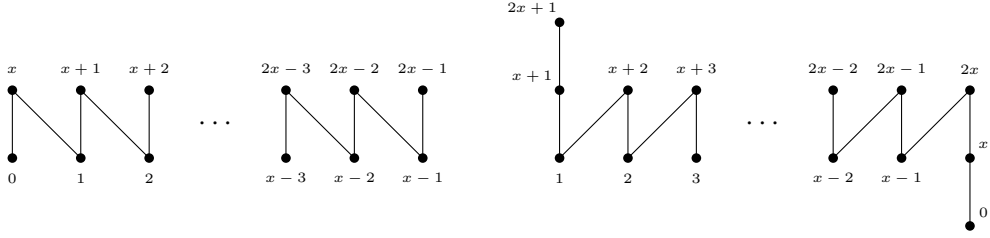
\begin{figure}
\caption{The perfect linear realizations~$\mathbf{g_1}$ and~$\mathbf{g_2}$ for $\{ (x-1)^{x-1} , x^{x} \}$ and $\{ x^{x+2} , (x+1)^{x-1} \}$ respectively.}\label{fig:xx+1}
\begin{center}
\begin{tikzpicture}[scale=0.9, every node/.style={transform shape}]

\fill (0,1) circle (2pt) ;
\fill (1,1) circle (2pt) ;
\fill (2,1) circle (2pt) ;
\fill (4,1) circle (2pt) ;
\fill (5,1) circle (2pt) ;
\fill (6,1) circle (2pt) ;

\fill (0,2) circle (2pt) ;
\fill (1,2) circle (2pt) ;
\fill (2,2) circle (2pt) ;
\fill (4,2) circle (2pt) ;
\fill (5,2) circle (2pt) ;
\fill (6,2) circle (2pt) ;

\draw (0,1) -- (0,2) -- (1,1) -- (1,2) -- (2,1) -- (2,2) ;
\draw (4,1) -- (4,2) -- (5,1) -- (5,2) -- (6,1) -- (6,2) ;

\node at (0, 0.7) {\tiny 0} ;           
\node at (1, 0.7) {\tiny 1} ;
\node at (2, 0.7) {\tiny 2} ;  
\node at (4, 0.7) {\tiny $x-3$} ;  
\node at (5, 0.7) {\tiny $x-2$} ;  
\node at (6, 0.7) {\tiny $x-1$} ;  

\node at (0, 2.3) {\tiny $x$} ;   
\node at (1, 2.3) {\tiny $x+1$} ;           
\node at (2, 2.3) {\tiny $x+2$} ;
\node at (4, 2.3) {\tiny $2x-3$} ;  
\node at (5, 2.3) {\tiny $2x-2$} ;  
\node at (6, 2.3) {\tiny $2x-1$} ;  

\node at (3, 1.5) {$\cdots$} ;

\fill (8,3) circle (2pt) ;
\fill (8,2) circle (2pt) ;
\fill (8,1) circle (2pt) ;
\fill (9,1) circle (2pt) ;
\fill (10,1) circle (2pt) ;
\fill (12,1) circle (2pt) ;
\fill (13,1) circle (2pt) ;

\fill (9,2) circle (2pt) ;
\fill (10,2) circle (2pt) ;
\fill (12,2) circle (2pt) ;
\fill (13,2) circle (2pt) ;
\fill (14,2) circle (2pt) ;
\fill (14,1) circle (2pt) ;
\fill (14,0) circle (2pt) ;

\draw (8,3) -- (8,1) -- (9,2) -- (9,1) -- (10,2) -- (10,1) ;
\draw (12,2) -- (12,1) -- (13,2) -- (13,1) -- (14,2) -- (14,0) ;

\node at (11, 1.5) {$\cdots$} ;

\node at (14.2, 0.2) {\tiny 0} ;
\node at (14.2, 1.2) {\tiny $x$} ;

\node at (8, 0.7) {\tiny 1} ;           
\node at (9, 0.7) {\tiny 2} ;
\node at (10, 0.7) {\tiny 3} ;  
\node at (12, 0.7) {\tiny $x-2$} ;  
\node at (13, 0.7) {\tiny $x-1$} ;  

\node at (7.6, 2.2) {\tiny $x+1$} ;           
\node at (7.6, 3.2) {\tiny $2x+1$} ;

\node at (9, 2.3) {\tiny $x+2$} ;
\node at (10, 2.3) {\tiny $x+3$} ;  
\node at (12, 2.3) {\tiny $2x-2$} ;  
\node at (13, 2.3) {\tiny $2x-1$} ;  
\node at (14, 2.3) {\tiny $2x$} ;

\end{tikzpicture}

\end{center}
\end{figure}

The construction for $\{ (x-1)^{x-1} , x^{x} \}$ in the proof of Lemma~\ref{lem:perf1} was first given in~\cite{OPPS}; the result for $\{ x^{x+2} , (x+1)^{x-1} \}$ seems to be new.  Using Lemma~\ref{lem:concat} we can concatenate these realizations to provide perfect realizations for more~2 and~3 (and more) element supports.  The construction idea of Lemma~\ref{lem:perf1} is explored further in~\cite{AO+}, including extensions to supports of the form~$\{x, kx \pm 1\}$ with~$k>1$.

Now let~$x \geq 2$.
Let $\omega(x,b)$ be the smallest possible number of 1-edges used in a standard linear realization with~$b$ $x$-edges and support~$\{1, x\}$.  By Lemma~\ref{lem:concat} and the discussion following it, the multiset~$\{1^a, x^b \}$ has a standard realization if and only if~$a \geq \omega(x,b)$.  We therefore want to determine~$\omega(x,b)$.  Theorem~\ref{th:omega} accomplishes this.  Some, but not all, of the constructions used in the proof are equivalent to those in~\cite{DJ09,HR09}; however those papers were not concerned with minimizing the number of 1-edges.

\begin{thm}\label{th:omega}{\rm ($\omega$-constructions)}
Let $x \geq 2$ and $b \geq 1$.  Let~$b = q'x + r'$ 
with $0 \leq r' < x$. Then~$\omega(x,b) = x-1$ in each of the following four situations: $x$ is even, $r'$ is even, $q'=0$, or $r'=1$.  Otherwise, $\omega(x,b) = x$.

The multiset~$\{ 1^a, x^b \}$ has a standard linear realization if and only if~$a \geq \omega(x,b)$.
\end{thm}

\begin{proof}
By Lemma~\ref{lem:fauxsets}, we have that $x-1$ is a lower bound for~$\omega(x,b)$.  A standard linear realization that meets this bound has $v = (x-1) + b + 1 = (q'+1)x + r'$; set $q = q'+1$ and~$r=r'$ so that $v = qx + r$ as usual.

It is also very limited in its structure: we have the minimum number of ways to hop from one fauxset to another and so all vertex labels in each fauxset~$\varphi_k$ must appear as the consecutive path~$\Psi_k$.  There are two potential constructions to consider, corresponding to whether we move to~$\varphi_1$ or~$\varphi_{x-1}$ from the initial fauxset~$\varphi_0$:
$$\mathbf{h_1} = \apph{}{k=0}{x-1} \Psi_k  {\rm \ \ \ and \ \ \ }  \mathbf{h_2} =  \Psi_0 \uplus \left( \apph{}{k=x-1}{1} \Psi_k \right) . $$
When do these form a valid Hamiltonian path?

Consider~$\mathbf{h_1}$.  This fails to be a Hamiltonian path if and only if we cannot make the hop from~$\varphi_{r-1}$ to~$\varphi_r$.   When~$ r=0$ this is not a consideration and the path is successful.  Otherwise, we can make this hop if and only if~$r$ is even as this is when~$\Psi_{r-1}$ is traversed in descending order.     

Consider~$\mathbf{h_2}$.  This fails to be a Hamiltonian path if and only if we cannot make the hop from~$\varphi_{r}$ to~$\varphi_{r-1}$.  If this hop is not required as part of the path then we are successful; this occurs if and only if~$r=1$. Assume~$r>1$.  We can make the hop if and only if~$x-r$ is odd (as this is when~$\Psi_{r}$ is traversed in descending order) or~$q'=0$ (as in this case~$\Psi_r$ is the single label~$r$).   We have that~$x-r$ is odd when~$x$ and~$r$ have opposite parities; in particular this path is successful when~$x$ is even and~$r$ is odd and hence for each even~$x$ one or other of the two paths is successful, depending on the parity of~$r$.

Recalling that~$r=r'$ for these paths, we see that we have established that $\omega(x,b) = x-1$ in the four cases given in the statement of the theorem.  Figure~\ref{fig:omega1} illustrates these four situations for~$(v,x,b) = (25, 7,18), (29, 8,21), (12, 7, 5), (22, 7,15)$.

\begin{figure}
\caption{The standard linear realizations~$\mathbf{h_1}$ (first diagram) and~$\mathbf{h_2}$ (remaining diagrams) of~$\{1^{x-1}, x^b\}$ for $(v,x,b) = (25, 7,18)$, $(29, 8,21)$, $(12, 7, 5)$ and $(22, 7,15)$.}\label{fig:omega1}
\begin{center}
\begin{tikzpicture}[scale=0.9, every node/.style={transform shape}]

\fill (1,1) circle (2pt) ;
\fill (1,2) circle (2pt) ;
\fill (1,5) circle (2pt) ;
\fill (1,6) circle (2pt) ;
\fill (1,7) circle (2pt) ;
\fill (1,8) circle (2pt) ;

\fill (2,1) circle (2pt) ;
\fill (2,2) circle (2pt) ;
\fill (2,5) circle (2pt) ;
\fill (2,6) circle (2pt) ;
\fill (2,7) circle (2pt) ;
\fill (2,8) circle (2pt) ;

\fill (3,1) circle (2pt) ;
\fill (3,2) circle (2pt) ;
\fill (3,5) circle (2pt) ;
\fill (3,6) circle (2pt) ;
\fill (3,7) circle (2pt) ;
\fill (3,8) circle (2pt) ;

\fill (4,1) circle (2pt) ;
\fill (4,2) circle (2pt) ;
\fill (4,5) circle (2pt) ;
\fill (4,6) circle (2pt) ;
\fill (4,7) circle (2pt) ;
\fill (4,8) circle (2pt) ;

\fill (5,1) circle (2pt) ;
\fill (5,5) circle (2pt) ;
\fill (5,6) circle (2pt) ;
\fill (5,7) circle (2pt) ;

\fill (6,1) circle (2pt) ;
\fill (6,5) circle (2pt) ;
\fill (6,6) circle (2pt) ;
\fill (6,7) circle (2pt) ;

\fill (7,0) circle (2pt) ;
\fill (7,1) circle (2pt) ;
\fill (7,5) circle (2pt) ;
\fill (7,6) circle (2pt) ;
\fill (7,7) circle (2pt) ;

\fill (9,1) circle (2pt) ;
\fill (9,2) circle (2pt) ;
\fill (9,3) circle (2pt) ;
\fill (9,5) circle (2pt) ;
\fill (9,6) circle (2pt) ;
\fill (9,7) circle (2pt) ;
\fill (9,8) circle (2pt) ;

\fill (10,1) circle (2pt) ;
\fill (10,2) circle (2pt) ;
\fill (10,3) circle (2pt) ;
\fill (10,5) circle (2pt) ;
\fill (10,6) circle (2pt) ;
\fill (10,7) circle (2pt) ;
\fill (10,8) circle (2pt) ;

\fill (11,1) circle (2pt) ;
\fill (11,2) circle (2pt) ;
\fill (11,3) circle (2pt) ;
\fill (11,5) circle (2pt) ;
\fill (11,6) circle (2pt) ;
\fill (11,7) circle (2pt) ;
\fill (11,8) circle (2pt) ;

\fill (12,1) circle (2pt) ;
\fill (12,2) circle (2pt) ;
\fill (12,3) circle (2pt) ;
\fill (12,5) circle (2pt) ;
\fill (12,6) circle (2pt) ;
\fill (12,7) circle (2pt) ;
\fill (12,8) circle (2pt) ;

\fill (13,1) circle (2pt) ;
\fill (13,2) circle (2pt) ;
\fill (13,3) circle (2pt) ;
\fill (13,5) circle (2pt) ;
\fill (13,6) circle (2pt) ;
\fill (13,7) circle (2pt) ;

\fill (14,1) circle (2pt) ;
\fill (14,2) circle (2pt) ;
\fill (14,3) circle (2pt) ;
\fill (14,5) circle (2pt) ;
\fill (14,6) circle (2pt) ;
\fill (14,7) circle (2pt) ;

\fill (15,0) circle (2pt) ;
\fill (15,1) circle (2pt) ;
\fill (15,2) circle (2pt) ;
\fill (15,3) circle (2pt) ;
\fill (15,5) circle (2pt) ;
\fill (15,6) circle (2pt) ;
\fill (15,7) circle (2pt) ;

\fill (16,4) circle (2pt) ;
\fill (16,5) circle (2pt) ;
\fill (16,6) circle (2pt) ;
\fill (16,7) circle (2pt) ;

\draw (1,5) -- (1,8) -- (2,8) -- (2,5) 
      -- (3,5) -- (3,8) -- (4,8) -- (4,5)    
     --  (5,5) -- (5,7) -- (6,7) -- (6,5) -- (7,5) -- (7,7);
          
\draw (16,4) -- (16,7) -- (15,7) -- (15,5) 
     -- (14,5) -- (14,7) -- (13,7)  -- (13,5)
     -- (12,5) -- (12,8) -- (11,8) -- (11,5)
     -- (10,5) -- (10,8) -- (9,8) -- (9,5)  ;
     
\draw (7,0) -- (7,1) -- (6,1) -- (5,1) -- (4,1) 
      -- (4,2) -- (3,2) -- (3,1) -- (2,1) -- (2,2) -- (1,2) -- (1,1) ;     
     
\draw (15,0) -- (15,3) -- (14,3) -- (14,1)  
     -- (13,1) -- (13,3) -- (12,3) -- (12,1)
     -- (11,1) -- (11,3) -- (10,3) -- (10,1) -- (9,1) -- (9,3) ;

\node at (0.7, 1) {\tiny 1} ;  
\node at (0.7, 2) {\tiny 8} ;  
\node at (0.7, 5) {\tiny 0} ;  
\node at (0.7, 6) {\tiny 7} ;  
\node at (0.7, 8) {\tiny 21} ;  

\node at (4.3, 2) {\tiny 11} ;  
\node at (4.3, 8) {\tiny 24} ;  

\node at (7.3, 0) {\tiny 0} ;  
\node at (7.3, 1) {\tiny 7} ;  
\node at (7.3, 5) {\tiny 6} ;  
\node at (7.3, 7) {\tiny 20} ;  

\node at (8.7, 1) {\tiny 1} ;  
\node at (8.7, 3) {\tiny 15} ;  
\node at (8.7, 5) {\tiny 1} ;  
\node at (8.7, 8) {\tiny 25} ;  

\node at (12.3, 8) {\tiny 28} ;

\node at (15.3, 0) {\tiny 0} ;  
\node at (15.3, 1) {\tiny 7} ;
\node at (15.3, 3) {\tiny 21} ;  
\node at (16.3, 4) {\tiny 0} ;  
\node at (16.3, 5) {\tiny 8} ;  
\node at (16.3, 7) {\tiny 24} ;  

\end{tikzpicture}
\end{center}
\end{figure}
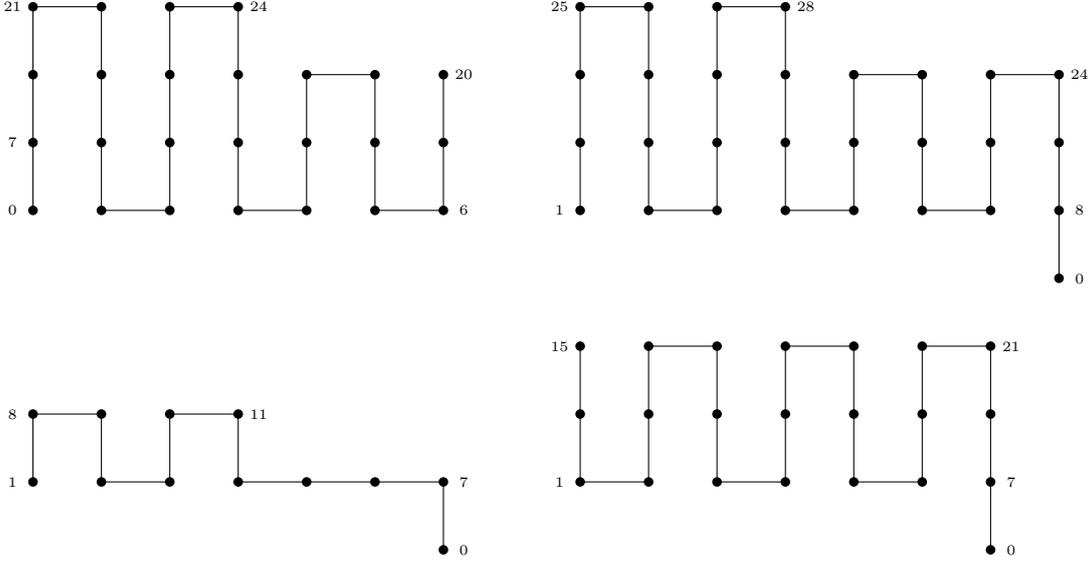

It remains to show that we can take~$\omega(x,b) = x$ in the uncovered cases.  Let~$x$ and~$r'$ be odd with~$r'>1$.  We construct a standard linear realization with
$$ v = x + b + 1 = (q'+1)x + r' + 1. $$
Set~$q = q'+1$ and $r = r'+1$ so that $v = qx + r$ as usual.
Let
$$\mathbf{h_3} = \Psi_0 \uplus \left( \apph{}{k=x-1}{3} \Psi_k \right) \uplus \Psi_2^{(q^*)} \uplus \Psi_1 \uplus \Psi_2^{(0, q^*-1)}. $$ 
First note that exactly one fauxset, $\varphi_{2}$, is visited twice and so if this is a Hamiltonian path then it realizes the correct multiset.  As~$r'$ is odd, $r$ is even.  As with~$\mathbf{h_2}$, we have~$x$ and~$r$ of different parity and so we can hop from $\varphi_r$ to~$\varphi_{r-1}$.  As~$r >3$, the three fauxsets~$\varphi_1$, $\varphi_2$ and~$\varphi_3$ each have~$q^* = q+1$ and so the second and third to last hops are successful.  

The first diagram of Figure~\ref{fig:omega2}  illustrates this for~$(v,x,b) = (25,7,17)$.

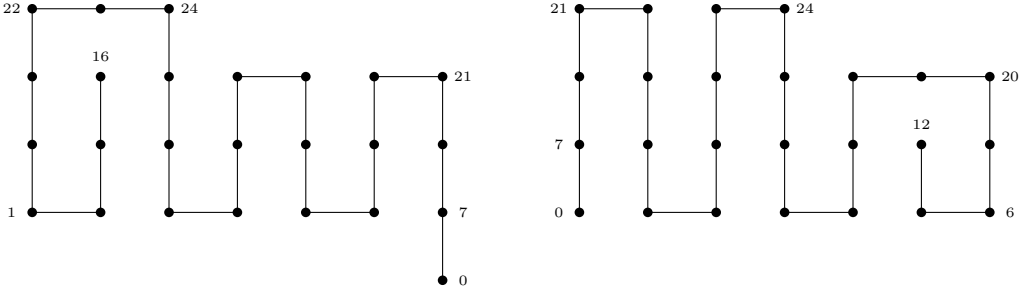
\begin{figure}
\caption{The standard linear realizations~$\mathbf{h_3}$ and~$\mathbf{h_4}$ of~$\{1^{x}, x^b\}$ for $(v,x,b) = (25, 7,17)$.}\label{fig:omega2}
\begin{center}
\begin{tikzpicture}[scale=0.9, every node/.style={transform shape}]

\fill (1,1) circle (2pt) ;
\fill (1,2) circle (2pt) ;
\fill (1,3) circle (2pt) ;
\fill (1,4) circle (2pt) ;

\fill (2,1) circle (2pt) ;
\fill (2,2) circle (2pt) ;
\fill (2,3) circle (2pt) ;
\fill (2,4) circle (2pt) ;

\fill (3,1) circle (2pt) ;
\fill (3,2) circle (2pt) ;
\fill (3,3) circle (2pt) ;
\fill (3,4) circle (2pt) ;

\fill (4,1) circle (2pt) ;
\fill (4,2) circle (2pt) ;
\fill (4,3) circle (2pt) ;

\fill (5,1) circle (2pt) ;
\fill (5,2) circle (2pt) ;
\fill (5,3) circle (2pt) ;

\fill (6,1) circle (2pt) ;
\fill (6,2) circle (2pt) ;
\fill (6,3) circle (2pt) ;

\fill (7,0) circle (2pt) ;
\fill (7,1) circle (2pt) ;
\fill (7,2) circle (2pt) ;
\fill (7,3) circle (2pt) ;

\fill (9,1) circle (2pt) ;
\fill (9,2) circle (2pt) ;
\fill (9,3) circle (2pt) ;
\fill (9,4) circle (2pt) ;

\fill (10,1) circle (2pt) ;
\fill (10,2) circle (2pt) ;
\fill (10,3) circle (2pt) ;
\fill (10,4) circle (2pt) ;

\fill (11,1) circle (2pt) ;
\fill (11,2) circle (2pt) ;
\fill (11,3) circle (2pt) ;
\fill (11,4) circle (2pt) ;

\fill (12,1) circle (2pt) ;
\fill (12,2) circle (2pt) ;
\fill (12,3) circle (2pt) ;
\fill (12,4) circle (2pt) ;

\fill (13,1) circle (2pt) ;
\fill (13,2) circle (2pt) ;
\fill (13,3) circle (2pt) ;

\fill (14,1) circle (2pt) ;
\fill (14,2) circle (2pt) ;
\fill (14,3) circle (2pt) ;

\fill (15,1) circle (2pt) ;
\fill (15,2) circle (2pt) ;
\fill (15,3) circle (2pt) ;

\draw  (7,0) -- (7,3) -- (6,3) -- (6,1) 
        -- (5,1) -- (5,3) -- (4,3) -- (4,1) 
        -- (3,1) -- (3,4) -- (1,4) -- (1,1) -- (2,1) -- (2,3) ;

\draw (9,1) -- (9,4) -- (10,4) -- (10,1)
     -- (11,1) -- (11,4) -- (12,4) -- (12,1)
     -- (13,1) -- (13,3) -- (14,3) -- (15,3) -- (15,1) -- (14,1) -- (14,2) ;
 
\node at (0.7, 1) {\tiny 1} ;  
\node at (0.7, 4) {\tiny 22} ;  
\node at (3.3, 4) {\tiny 24} ;  
\node at (2, 3.3) {\tiny 16} ;  
\node at (7.3, 0) {\tiny 0} ;  
\node at (7.3, 1) {\tiny 7} ;  
\node at (7.3, 3) {\tiny 21} ;  

\node at (8.7, 1) {\tiny 0} ;  
\node at (8.7, 2) {\tiny 7} ;  
\node at (8.7, 4) {\tiny 21} ;  
\node at (12.3, 4) {\tiny 24} ;  
\node at (14, 2.3) {\tiny 12} ;  
\node at (15.3, 1) {\tiny 6} ;  
\node at (15.3, 3) {\tiny 20} ;  

\end{tikzpicture}\end{center}
\end{figure}

The final statement of the theorem follows from Lemma~\ref{lem:ones}
\end{proof}

As an alternative to~$\mathbf{h_3}$ in the proof of Theorem~\ref{th:omega}, which moves through the fauxsets in decreasing order, we may use the realization~$\mathbf{h_4}$, which moves through the fauxsets in increasing order.
If~$r \neq x-1$, let
$$\mathbf{h_4} = \left( \apph{}{k=0}{x-3} \Psi_k  \right) \uplus \Psi_{x-2}^{(q^*,q^*)} \uplus \Psi_{x-1} \uplus \Psi_{x-2}^{(0,q^*-1)};$$     
if~$r = x-1$ let
$$\mathbf{h_4}  = \left( \apph{}{k=0}{x-2} \Psi_k  \right) \uplus \Psi_{x-2}^{(q^* - 1,q^*)} \uplus  \Psi_{x-1} \uplus \Psi_{x-2}^{(0,q^*-2)}.$$     
This is illustrated for~$(v,x,b) = (25,7,17)$ in the second diagram of Figure~\ref{fig:omega2}.  

The paths~$\mathbf{h_3}$ and~$\mathbf{h_4}$ can be thought of as adaptations of~$\mathbf{h_2}$ and~$\mathbf{h_1}$ respectively, obtained by passing through the penultimate fauxset using only one label and then returning to it after the final fauxset is traversed.  This trick, which replaces an~$x$ in the realized multiset with a~1, is frequently useful and we refer to it informally as a {\em tail curl}.  These four realizations are used throughout the paper; we call them {\em $\omega$-constructions}. 

We arrive at the following corollary, which improves items~6--10 of Theorem~\ref{th:known}.  

\begin{cor}\label{cor:omega_gen}{\rm ($\omega$-concatenations)}
There is a linear realization for $L = \{ 1^a, x^b, y^c \}$ when $a \geq x+y-\epsilon$, where $\epsilon$ is the number of elements of $\{x, y\}$ that are even.  
\end{cor}

\begin{proof}
Write~$L = \{1^{\omega(x,b)} , x^b \} \cup \{ 1^{a - \omega(x,b)} , y^c \}$.  We have $a - \omega(x,b) \geq \omega(y,c)$ and so each multiset in the union has a standard linear realization by Theorem~\ref{th:omega}.  Therefore~$L$ has a linear realization by Lemma~\ref{lem:concat}.
\end{proof}

By considering the subcases of Theorem~\ref{th:omega}, Corollary~\ref{cor:omega_gen} could be made slightly stronger and further improvements are possible using small variations of the construction.  We do not explore this fully here.  However, in Section~\ref{sec:evenx} we give one example of this as it gives a result that forms a step toward the main results.

\section{The case $U = \{1,x, x+1\} $}\label{sec:1xx+1}

The main goal in this section is to construct a grid-based linear realization for each multiset  $\{ 1^a, x^b, (x+1)^c \}$ with~$a \geq x+1$.  The bulk of the work is divided into two lemmas: Lemma~\ref{lem:smallb} covers cases with small~$b$ and Lemma~\ref{lem:smallc} covers cases with small~$c$.  The final step, taken in Theorem~\ref{th:1xx+1}, combines these results with the perfect realizations of Lemma~\ref{lem:perf1} to give the full result.  
All of the realizations constructed are standard.  An immediate consequence of this for multisets with support larger than~3 is given in Corollary~\ref{cor:bigsupp}.

In~\cite{HR09}, the first paper on the BHR Conjecture, Horak and Rosa introduced $\alpha$-~and $\beta$-transformations, or $\alpha$-~and $\beta$-moves, as methods of altering a Hamiltonian path through~$K_v$ in such a way that the realized multiset does not change very much.  While we do not use them explicitly in this paper (although the tail curl is a type of~$\alpha$-move) we use a different but similar idea in the lemmas of this section.  We call it a~$\gamma$-move to fit with Horak and Rosa's naming scheme.  

Given a Hamiltonian path in~$K_v$, a {\em $\gamma$-move} removes one vertex from the path and re-inserts it somewhere else.  This changes the realized multiset by at most three elements.  We shall move the vertex only small distances in terms of a grid-graph representation of the Hamiltonian path, allowing us to think of the $\gamma$-move as having a local effect, even though the vertex might move a considerable distance in terms of the Hamiltonian path itself.  The method of perturbing a Hamiltonian path in such a way that the effect is local when the vertices are arranged in a grid is similar to some ideas of~\cite[Section~2]{Nishat20}.

The $\gamma$-moves we use only change at most~2 of the elements of the realized multiset.   We use one type in Lemma~\ref{lem:smallb} and another in Lemma~\ref{lem:smallc}.  Geometrically, they are closely related: one is a $180^{\circ}$ rotation of the other in the grid-based graph representations we use.  Their reflections may be used in a similar way, but we do not need them here.

\begin{lem}\label{lem:smallb}{\rm (Small $b$)} 
Let $L = \{ 1^a, x^b, (x+1)^c \}$ with~$b < x$.  There is a standard linear realization for~$L$ whenever $a \geq \omega(x+1,b+c)$.  In particular, there is one when~$a \geq  x+1$.
\end{lem}

\begin{proof}
The method is to start with the $\omega$-construction of a linear realization for the multiset $\{ 1^{\omega(x+1,b+c)}, (x+1)^{b+c} \}$ and replace~$b$ of the $(x+1)$-edges with~$x$-edges.  During this process the number of~$1$-edges also sometimes decreases.
We may then concatenate additional 1-edges as required at the start.  As~$\omega(x+1,b+c) \leq x+1$ this is sufficient to prove the result.  

The main method by which we make the edge replacements keeps the number of 1-edges fixed, increases the number of $x$-edges by~2 and reduces the number of~$(x+1)$-edges by 2.  

Suppose that, for some~$g$, a standard linear realization has the subsequences~$(g, g+x+1)$ and~$(g+x+2, g+1, g+2)$.  These realize the edge lengths~$\{ 1, (x+1)^2 \}$.  We may make a $\gamma$-move taking the~$g+1$ vertex from the center of the second subsequence to the first without disturbing the remainder of the realization. This gives subsequences~$(g, g+1, g+x+1)$ and~$(g+x+2, g+2)$, which together realize~$\{ 1, x^2 \}$.  Thus given a realization for some multiset~$L$ that has such subsequences we may make this switch to obtain a realization for~$(L \cup \{x^2\}) \setminus \{(x+1)^2\}$.    For the purposes of this proof, call this a {\em $\gamma$-move at~$g$} and don't consider any other $\gamma$-moves.

Further, if there are~$t$ values at which we can perform such a~$\gamma$-move we may obtain a realization for each~$(L \cup \{x^{2i}\}) \setminus \{(x+1)^{2i}\}$ for each~$i$ in the range~$1 \leq i \leq t$.

Clearly, these $\gamma$-moves can only exchange an even number of~$(x+1)$-edges for $x$-edges; we need an additional tool when~$b$ is odd.  Suppose that a realization for a multiset~$L$ of size~$v-1$ has subsequence~$(v-2, v-1, v-x-2)$.  This realizes~$\{1,x+1\}$.   We may omit~$v-1$ (the largest vertex label) to get subsequence $(v-2,v-x-2)$, which realizes~$\{x\}$.  The new sequence realizes~$(L \cup \{x\}) \setminus \{ 1, x+1 \}$.   Call this process a {\em corner cut}.

Write $b+c = q(x+1) + r$ with~$0 \leq r < x+1$.   
Exactly how we apply $\gamma$-moves and corner cuts splits into five cases, four with~$b+c \geq x+1$ corresponding to the options for the parities of~$r$ and~$x+1$ and a fifth for~$b+c < x+1$.  By Theorem~\ref{th:omega}, we assume~$b,c >0$ throughout.

Case 1: $b+c \geq x+1$,~$r$ even and~$x+1$ even.  We start with the $\omega$-construction~$\mathbf{h_1}$.  We may perform $\gamma$-moves at any (or all) of~$0,2,4,\ldots, x-3$ and we may perform a corner cut.   If~$b$ is even, perform $b/2$ $\gamma$-moves; if~$b$ is odd, perform~$(b-1)/2$ $\gamma$-moves and a corner cut.  For~$x+1=8$ and~$b+c = 20$, the first diagram of Figure~\ref{fig:41a} illustrates~$b=3$, using a corner cut and a~$\gamma$-move at~0, and the second illustrates the maximum required~$b = 6$, using~$\gamma$-moves at~0,~2  and~4.

\begin{figure}[tp]
\caption{Standard linear realizations for~$\{1^6, 7^3, 8^{17}\}$ and~$\{1^7, 7^6, 8^{14}\}$.}\label{fig:41a}
\begin{center}
\begin{tikzpicture}[scale=0.8, every node/.style={transform shape}]

\fill (0,1) circle (2pt) ;
\fill (0,2) circle (2pt) ;
\fill (0,3) circle (2pt) ;
\fill (0,4) circle (2pt) ;

\fill (1,1) circle (2pt) ;
\fill (1,2) circle (2pt) ;
\fill (1,3) circle (2pt) ;
\fill (1,4) circle (2pt) ;

\fill (2,1) circle (2pt) ;
\fill (2,2) circle (2pt) ;
\fill (2,3) circle (2pt) ;
\fill (2,4) circle (2pt) ;

\fill (3,1) circle (2pt) ;
\fill (3,2) circle (2pt) ;
\fill (3,3) circle (2pt) ;
\fill (3,4) circle (2pt) ;

\fill (4,1) circle (2pt) ;
\fill (4,2) circle (2pt) ;
\fill (4,3) circle (2pt) ;

\fill (5,1) circle (2pt) ;
\fill (5,2) circle (2pt) ;
\fill (5,3) circle (2pt) ;

\fill (6,1) circle (2pt) ;
\fill (6,2) circle (2pt) ;
\fill (6,3) circle (2pt) ;

\fill (7,1) circle (2pt) ;
\fill (7,2) circle (2pt) ;
\fill (7,3) circle (2pt) ;

\draw (0,1) -- (1,1) -- (0,2) -- (0,4) -- (1,4) -- (1,2) -- (2,1) -- (2,4) -- (3,3) -- (3,1) ;
\draw (3,1)  -- (4,1) -- (4,3) -- (5,3) -- (5,1) -- (6,1) -- (6,3) -- (7,3) -- (7,1);

\node at (-0.3, 1) {\tiny 0} ;  
\node at (-0.3, 2) {\tiny 8} ;  
\node at (-0.3, 4) {\tiny 24} ;  
\node at (2.3, 4) {\tiny 26} ;  
\node at (7.3, 1) {\tiny 7} ;  
\node at (7.3, 3) {\tiny 23} ;   

\fill (9,1) circle (2pt) ;
\fill (9,2) circle (2pt) ;
\fill (9,3) circle (2pt) ;
\fill (9,4) circle (2pt) ;

\fill (10,1) circle (2pt) ;
\fill (10,2) circle (2pt) ;
\fill (10,3) circle (2pt) ;
\fill (10,4) circle (2pt) ;

\fill (11,1) circle (2pt) ;
\fill (11,2) circle (2pt) ;
\fill (11,3) circle (2pt) ;
\fill (11,4) circle (2pt) ;

\fill (12,1) circle (2pt) ;
\fill (12,2) circle (2pt) ;
\fill (12,3) circle (2pt) ;
\fill (12,4) circle (2pt) ;

\fill (13,1) circle (2pt) ;
\fill (13,2) circle (2pt) ;
\fill (13,3) circle (2pt) ;

\fill (14,1) circle (2pt) ;
\fill (14,2) circle (2pt) ;
\fill (14,3) circle (2pt) ;

\fill (15,1) circle (2pt) ;
\fill (15,2) circle (2pt) ;
\fill (15,3) circle (2pt) ;

\fill (16,1) circle (2pt) ;
\fill (16,2) circle (2pt) ;
\fill (16,3) circle (2pt) ;

\draw (9,1) -- (10,1) -- (9,2) -- (9,4) -- (10,4) -- (10,2) -- (11,1) -- (12,1) -- (11,2) -- (11,4) -- (12,4) -- (12,2) 
 -- (13,1) -- (14,1) -- (13,2) -- (13,3) -- (14,3) -- (14,2) -- (15,1) -- (15,3) -- (16,3) -- (16,1)  ;

\node at (8.7, 1) {\tiny 0} ;  
\node at (8.7, 2) {\tiny 8} ;  
\node at (8.7, 4) {\tiny 24} ;  
\node at (12.3, 4) {\tiny 27} ;  
\node at (16.3, 1) {\tiny 7} ;  
\node at (16.3, 3) {\tiny 23} ;

\end{tikzpicture}\end{center}
\end{figure}
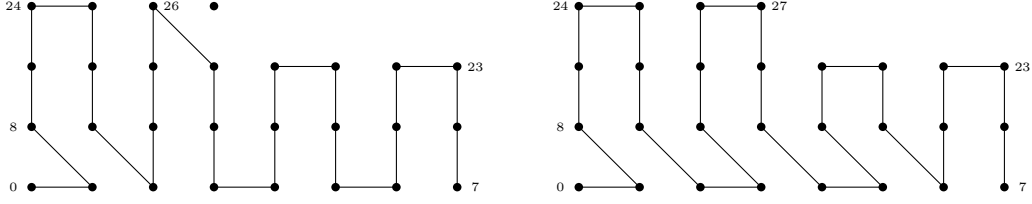

Case 2: $b+c \geq x+1$,~$r$ even and~$x+1$ odd. This is very similar to Case~1.  
We again use~$\mathbf{h_1}$ and now have $\gamma$-moves available at~$0,2,4,\ldots, x-4$.  This allows us to replace any even number of~$(x+1)$-edges with~$x$-edges up to a maximum of~$x-2$.  A corner cut is available to allow us to make one additional replacement to cover odd values of~$b$, except when~$r=0$.   When~$r=0$ we may use a slight variation of the corner cut.  The realization~$\mathbf{h_1}$ has subsequence~$(v-3, v-2, v-x-3)$, which realizes~$\{1, x+1\}$, and final label~$v-1$ (and this remains the case after applying any number of the~$\gamma$-moves).  Remove~$v-2$ from the subsequence and attach it at the end, adjacent to~$v-1$.  The resulting sequence realizes~$(L \cup \{ 1,x\} ) \setminus \{1,x+1\} = (L \cup \{ x\} ) \setminus \{ x+1\}$.  The first diagram of Figure~\ref{fig:41b} shows the construction for~$x+1 = 7$,~$r=0$ and~$b=5$, which requires the modified corner cut.

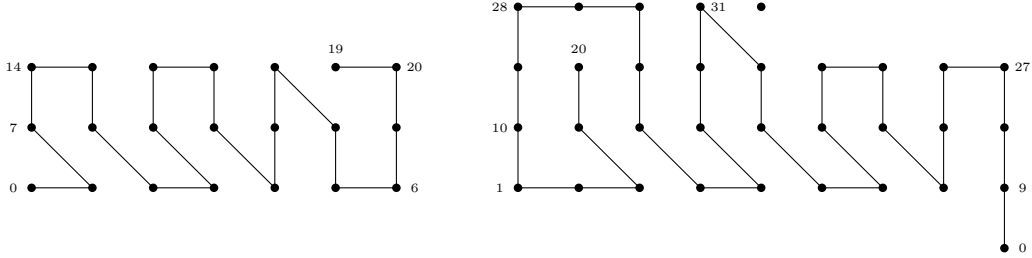
\begin{figure}[tp]
\caption{Standard linear realizations for~$\{1^5, 6^5, 7^{9}\}$ and~$\{1^8, 8^7, 9^{15}\}$.}\label{fig:41b}
\begin{center}
\begin{tikzpicture}[scale=0.8, every node/.style={transform shape}]

\fill (0,1) circle (2pt) ;
\fill (0,2) circle (2pt) ;
\fill (0,3) circle (2pt) ;

\fill (1,1) circle (2pt) ;
\fill (1,2) circle (2pt) ;
\fill (1,3) circle (2pt) ;

\fill (2,1) circle (2pt) ;
\fill (2,2) circle (2pt) ;
\fill (2,3) circle (2pt) ;

\fill (3,1) circle (2pt) ;
\fill (3,2) circle (2pt) ;
\fill (3,3) circle (2pt) ;

\fill (4,1) circle (2pt) ;
\fill (4,2) circle (2pt) ;
\fill (4,3) circle (2pt) ;

\fill (5,1) circle (2pt) ;
\fill (5,2) circle (2pt) ;
\fill (5,3) circle (2pt) ;

\fill (6,1) circle (2pt) ;
\fill (6,2) circle (2pt) ;
\fill (6,3) circle (2pt) ;

\draw (0,1) -- (1,1) -- (0,2) -- (0,3) -- (1,3) -- (1,2) -- (2,1)  -- (3,1) -- (2,2) -- (2,3) -- (3,3) -- (3,2) 
   -- (4,1) -- (4,3) -- (5,2) -- (5,1) -- (6,1) -- (6,3) -- (5,3) ;

\node at (-0.3, 1) {\tiny 0} ;  
\node at (-0.3, 2) {\tiny 7} ;  
\node at (-0.3, 3) {\tiny 14} ;    
\node at (6.3, 1) {\tiny 6} ;  
\node at (6.3, 3) {\tiny 20} ;   
\node at (5, 3.3) {\tiny 19} ;  

\fill (8,1) circle (2pt) ;
\fill (8,2) circle (2pt) ;
\fill (8,3) circle (2pt) ;
\fill (8,4) circle (2pt) ;

\fill (9,1) circle (2pt) ;
\fill (9,2) circle (2pt) ;
\fill (9,3) circle (2pt) ;
\fill (9,4) circle (2pt) ;

\fill (10,1) circle (2pt) ;
\fill (10,2) circle (2pt) ;
\fill (10,3) circle (2pt) ;
\fill (10,4) circle (2pt) ;

\fill (11,1) circle (2pt) ;
\fill (11,2) circle (2pt) ;
\fill (11,3) circle (2pt) ;
\fill (11,4) circle (2pt) ;

\fill (12,1) circle (2pt) ;
\fill (12,2) circle (2pt) ;
\fill (12,3) circle (2pt) ;
\fill (12,4) circle (2pt) ;

\fill (13,1) circle (2pt) ;
\fill (13,2) circle (2pt) ;
\fill (13,3) circle (2pt) ;

\fill (14,1) circle (2pt) ;
\fill (14,2) circle (2pt) ;
\fill (14,3) circle (2pt) ;

\fill (15,1) circle (2pt) ;
\fill (15,2) circle (2pt) ;
\fill (15,3) circle (2pt) ;

\fill (16,0) circle (2pt) ;
\fill (16,1) circle (2pt) ;
\fill (16,2) circle (2pt) ;s
\fill (16,3) circle (2pt) ;

\draw (10,4) -- (8,4) -- (8,1)   -- (9,1) -- (10,1) -- (9,2) -- (9,3)  ;
\draw  (10,4) -- (10,2) -- (11,1) -- (12,1) -- (11,2) -- (11,4) -- (12,3) -- (12,2) 
 -- (13,1) -- (14,1) -- (13,2) -- (13,3) -- (14,3) -- (14,2) -- (15,1) -- (15,3) -- (16,3) -- (16,0)  ;

\node at (7.7, 1) {\tiny 1} ;  
\node at (7.7, 2) {\tiny 10} ;  
\node at (7.7, 4) {\tiny 28} ;  
\node at (9, 3.3) {\tiny 20} ;
\node at (11.3, 4) {\tiny 31} ;  
\node at (16.3, 0) {\tiny 0} ;  
\node at (16.3, 1) {\tiny 9} ;  
\node at (16.3, 3) {\tiny 27} ;  

\end{tikzpicture}\end{center}
\end{figure}

Case 3: $b+c \geq x+1$,~$r$ odd and~$x+1$ even.   Start with the $\omega$-construction~$\mathbf{h_2}$ for $\{1^x, (x+1)^{b+c} \}$  (the second diagram of Figure~\ref{fig:omega1} illustrates this for~$r=5$ and~$x+1=8$).  We have $\gamma$-moves available at~$1, 3, \ldots, x-2$ and a corner cut is always possible.  This gives sufficient options to cover all possibilities for~$b$.  

Case 4: $b+c \geq x+1$,~$r$ odd and~$x+1$ odd.  This is the case where the $\omega$-construction can require~$x+1$ 1-edges, implemented via a tail curl.    We use the realization~$\mathbf{h_3}$, which is illustrated for~$x+1 = 7$ and~$r=3$ in the first diagram of Figure~\ref{fig:omega2}.  We have $\gamma$-moves available at~$2, 4, \ldots, x-2$ and a corner cut is always possible.  This gives sufficient options to cover all possibilities for~$b$.  The second diagram of Figure~\ref{fig:41b} shows the construction for~$x+1 = 9$, $r = 5$ and~$b=7$, which requires all available $\gamma$-moves and a corner cut.

Case 5: $b+c < x+1$.   If~$r$ is even, we use~$\mathbf{h_1}$ (which we may regardless of the parity of~$x+1$, as~$q=0$) and the argument goes through as in Case~1 above with~$r$ even and~$x+1$ even.    If~$r$ is odd, we use~$\mathbf{h_2}$ (which we may regardless of the parity of~$x+1$, as~$q=0$) and the argument goes through as in Case~3 above with~$r$ odd and~$x+1$ even.  
\end{proof}

The next lemma covers situations when~$c$ is small.  There is some overlap with Lemma~\ref{lem:smallb} when both~$b$ and~$c$ are small.

\begin{lem}\label{lem:smallc}{\rm (Small $c$)} 
Let $L = \{ 1^a, x^b, (x+1)^c \}$ where~$c < x+1$.  There is a standard linear realization for~$L$ whenever $a \geq \omega(x,b+c) + 1$.  In particular, there is one when~$a \geq  x+1$.

\end{lem}

\begin{proof}
The approach is similar to that of Lemma~\ref{lem:smallb}.  Here we start with the $\omega$-construction of a linear realization for~$\{1^{\omega(x, b+c)} , x^{b+c} \}$ and replace~$c$ of the~$x$-edges with~$(x+1)$-edges.  This process might cause the number of 1-edges to change including, unlike Lemma~\ref{lem:smallb}, an increase in some cases.  However, in all cases, the number of 1-edges in the resulting realization is at most~$x+1$; we may concatenate additional 1-edges as required at the start.

The main method by which we make the edge replacements keeps the number of 1-edges fixed, increases the number of $(x+1)$-edges by~2 and reduces the number of~$x$-edges by 2.  

We again use a~$\gamma$-move.  This time we take subsequences $(g-x,g,g+1)$ and $(g-1,g-x-1)$ and move the vertex~$g$ to the middle of the second subsequence to give subsequences $(g-x,g+1)$ and $(g-1,g, g-x-1)$.   If the original sequence realizes~$L$ then the new sequence realizes~$(L \cup \{ (x+1)^2 \} ) \setminus \{ x^2 \}$.  For the purposes of this proof, call this process a {\em $\gamma$-move at~$g$}.  If there are~$t$ values at which we can perform such a~$\gamma$-move we may obtain a realization for each~$(L \cup \{(x+1)^{2i}\}) \setminus \{x^{2i}\}$ for each~$i$ in the range~$1 \leq i \leq t$.

To acquire an odd number of~$(x+1)$-edges we use a {\em corner flap}:   Suppose that a realization for a multiset~$L$ of size~$v-1$ has the vertex labels~$v-1$ and~$v-x-1$ adjacent to each other, giving a difference of~$x$.  We may insert vertex~$v$ between them, creating a realization for the multiset~$(L \cup \{1, x+1 \}) \setminus \{x\}$ of size~$v$.

The first diagram (when ignoring the red dashed line) in Figure~\ref{fig:42a} shows a realization for~$\{1^8, 8^{19}, 9^3 \}$ produced from the optimal realization~$\mathbf{h_1}$ for~$\{1^7, 8^{22} \}$ by using a $\gamma$-move at~26 and adding a corner flap.

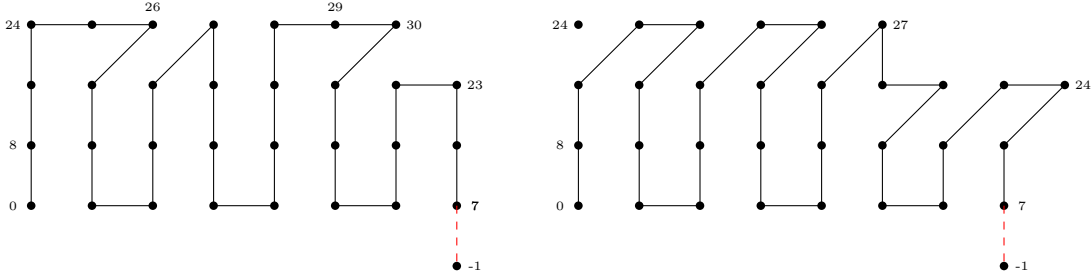
\begin{figure}[tp]
\caption{Realizations for~$\{1^8, 8^{19}, 9^3 \}$ and~$\{1^7, 8^{14}, 9^8\}$ (ignoring the red dashed line) and precursors to realizations for~$\{1^8, 8^{20}, 9^3 \}$ and~$\{1^7, 8^{15}, 9^8\}$ (red dashed line included).}\label{fig:42a}
\begin{center}
\begin{tikzpicture}[scale=0.8, every node/.style={transform shape}]

\fill (0,1) circle (2pt) ;
\fill (0,2) circle (2pt) ;
\fill (0,3) circle (2pt) ;
\fill (0,4) circle (2pt) ;

\fill (1,1) circle (2pt) ;
\fill (1,2) circle (2pt) ;
\fill (1,3) circle (2pt) ;
\fill (1,4) circle (2pt) ;

\fill (2,1) circle (2pt) ;
\fill (2,2) circle (2pt) ;
\fill (2,3) circle (2pt) ;
\fill (2,4) circle (2pt) ;

\fill (3,1) circle (2pt) ;
\fill (3,2) circle (2pt) ;
\fill (3,3) circle (2pt) ;
\fill (3,4) circle (2pt) ;

\fill (4,1) circle (2pt) ;
\fill (4,2) circle (2pt) ;
\fill (4,3) circle (2pt) ;
\fill (4,4) circle (2pt) ;

\fill (5,1) circle (2pt) ;
\fill (5,2) circle (2pt) ;
\fill (5,3) circle (2pt) ;
\fill (5,4) circle (2pt) ;

\fill (6,1) circle (2pt) ;
\fill (6,2) circle (2pt) ;
\fill (6,3) circle (2pt) ;
\fill (6,4) circle (2pt) ;

\fill (7,0) circle (2pt) ;
\fill (7,1) circle (2pt) ;
\fill (7,2) circle (2pt) ;
\fill (7,3) circle (2pt) ;

\draw (0,1) -- (0,4) -- (2,4) -- (1,3) -- (1,1) -- (2,1) -- (2,3) -- (3,4) -- (3,1) -- (4,1) --
              (4,4) -- (6,4) -- (5,3) -- (5,1) -- (6,1) -- (6,3) -- (7,3) -- (7,1) ;
\draw[red, dashed] (7,1) -- (7,0);

\node at (-0.3, 1) {\tiny 0} ;  
\node at (-0.3, 2) {\tiny 8} ;  
\node at (-0.3, 4) {\tiny 24} ;  
\node at (2, 4.3) {\tiny 26} ;
\node at (5, 4.3) {\tiny 29} ;
\node at (6.3, 4) {\tiny 30} ;   
\node at (7.3, 1) {\tiny 7} ;  
\node at (7.3, 1) {\tiny 7} ;  
\node at (7.3, 0) {\tiny -1} ;  
\node at (7.3, 3) {\tiny 23} ;  

\fill (9,1) circle (2pt) ;
\fill (9,2) circle (2pt) ;
\fill (9,3) circle (2pt) ;
\fill (9,4) circle (2pt) ;

\fill (10,1) circle (2pt) ;
\fill (10,2) circle (2pt) ;
\fill (10,3) circle (2pt) ;
\fill (10,4) circle (2pt) ;

\fill (11,1) circle (2pt) ;
\fill (11,2) circle (2pt) ;
\fill (11,3) circle (2pt) ;
\fill (11,4) circle (2pt) ;

\fill (12,1) circle (2pt) ;
\fill (12,2) circle (2pt) ;
\fill (12,3) circle (2pt) ;
\fill (12,4) circle (2pt) ;

\fill (13,1) circle (2pt) ;
\fill (13,2) circle (2pt) ;
\fill (13,3) circle (2pt) ;
\fill (13,4) circle (2pt) ;

\fill (14,1) circle (2pt) ;
\fill (14,2) circle (2pt) ;
\fill (14,3) circle (2pt) ;
\fill (14,4) circle (2pt) ;

\fill (15,1) circle (2pt) ;
\fill (15,2) circle (2pt) ;
\fill (15,3) circle (2pt) ;

\fill (16,0) circle (2pt) ;
\fill (16,1) circle (2pt) ;
\fill (16,2) circle (2pt) ;
\fill (16,3) circle (2pt) ;

\fill (17,3) circle (2pt) ;

\draw (9,1) -- (9,3) -- (10,4) -- (11,4) -- (10,3) -- (10,1) -- (11,1) -- (11,3) -- (12,4) -- (13,4) 
          -- (12,3) -- (12,1) -- (13,1) -- (13,3) -- (14,4) -- (14,3) -- (15,3) -- (14,2) -- (14,1) 
          -- (15,1) -- (15,2) -- (16,3) -- (17,3) -- (16,2) -- (16,1) ;
\draw[red, dashed] (16,1) -- (16,0);

\node at (8.7, 1) {\tiny 0} ;  
\node at (8.7, 2) {\tiny 8} ;  
\node at (8.7, 4) {\tiny 24} ;  
\node at (14.3, 4) {\tiny 27} ;  
\node at (16.3, 0) {\tiny -1} ;  
\node at (16.3, 1) {\tiny 7} ;  
\node at (17.3, 3) {\tiny 24} ;  

\end{tikzpicture}\end{center}
\end{figure}

Write~$b+c = qx+r$ with~$0 \leq r < x$.  By Theorem~\ref{th:omega}, we assume~$b,c >0$ throughout.  Similarly to Lemma~\ref{lem:smallb}, we have five cases: four depending on the parities of~$x$ and~$r$ when~$b+c \geq x$ (although there are some cross-case situations for small~$r$) and one for~$b+c < x+1$.

Case 1: $b+c \geq x$ with $x$ and~$r$ even.  Start with the realization~$\mathbf{h_1}$.  Provided~$b+c > x$ we may perform~$\gamma$-moves at any (or all) of $v-2, v-4, \ldots v - x$.  This allows us to replace any even number of~$x$-edges with~$(x+1)$-edges up to a maximum of~$x$.   We can also use a corner flap, so odd value of~$b$ are covered (in fact, this lets us make up to~$x+1$ exchanges; more than we need here but this fact is useful in the next case).  The second diagram (ignoring the red dashed line) illustrates the realization for~$\{1^7, 8^{14}, 9^8\}$ obtained from the realization~$\mathbf{h_1}$ for~$\{1^7, 8^{22}\}$ using all~$8/2 = 4$ $\gamma$-moves.  (Note that the vertex labelled~24 is drawn in a different place to make the resulting realization visually clearer and maintain the correspondence between edge orientation and length.)

When~$b+c = x$ we may perform $\gamma$-moves at any (or all) of~$v-2, v-4, \ldots v - x+2$ and also use a corner flap.  The only possible missing value is~$c = x$.  However, in this situation we have~$b=0$.

Case 2: $b+c \geq x$ with $x$ even and~$r$ odd.  We can easily build the required realizations for this case from the previous one.  Observe that the final label of the realization obtained in the case when both~$x$ and~$r$ are even is~$x-1$.    Add a new vertex labeled~$-1$ and attach it at the end using an~$x$-edge. Add~1 to all of the vertex labels.  This process adds one to~$b+c$ and hence~$1$ to~$r$, except when~$r = x-1$ in which case it becomes~$0$.    This method covers everything except~$b+c = x+1$ with~$c = x$. It is possible to use $\gamma$-moves on the realization~$\mathbf{h_2}$, but it is simpler to observe that we have~$b=1 <x $ and then use Lemma~\ref{lem:smallb}.

Figure~\ref{fig:42a} with the red dashed lines included illustrates the constructions for~$\{1^8, 8^{20}, 9^3 \}$ and~$\{1^7, 8^{15}, 9^8\}$ prior to adding~1 to the vertex labels.

Case 3: $b+c \geq x$ with odd $x$ and even~$r \neq 0$ or $r=1$.  When~$t$ is even we start with the $\omega$-construction~$\mathbf{h_2}$.  We may use a corner flap and $\gamma$-moves at any (or all) of $v-x, v-x+2, \ldots, v-r-1$ and $v-r+2, v-r+4, \ldots, v-2$.  This is sufficient to replace up to~$x$ $x$-edges with~$(x+1)$-edges as required.  The first diagram of Figure~\ref{fig:42b} illustrates this with the maximum number of $\gamma$-moves and a corner flap for~$x=7$ and~$r=4$.  The same construction works for~$r=1$, with the $\gamma$-moves all happening at the highest level, in positions $v-x+1, v-x+3, \ldots, v-2$.  This is illustrated in the second diagram of Figure~\ref{fig:42b} with the maximum number of~$\gamma$-moves and a corner flap for~$x=7$.

\begin{figure}[tp]
\caption{Standard linear realizations for~$\{1^7, 7^{11}, 8^{7}\}$ and~$\{1^7, 7^8, 8^{7}\}$.}\label{fig:42b}
\begin{center}
\begin{tikzpicture}[scale=0.9, every node/.style={transform shape}]

\fill (0,1) circle (2pt) ;
\fill (0,2) circle (2pt) ;
\fill (0,3) circle (2pt) ;
\fill (0,4) circle (2pt) ;

\fill (1,1) circle (2pt) ;
\fill (1,2) circle (2pt) ;
\fill (1,3) circle (2pt) ;
\fill (1,4) circle (2pt) ;

\fill (2,1) circle (2pt) ;
\fill (2,2) circle (2pt) ;
\fill (2,3) circle (2pt) ;
\fill (2,4) circle (2pt) ;

\fill (3,1) circle (2pt) ;
\fill (3,2) circle (2pt) ;
\fill (3,3) circle (2pt) ;
\fill (3,4) circle (2pt) ;

\fill (4,1) circle (2pt) ;
\fill (4,2) circle (2pt) ;
\fill (4,3) circle (2pt) ;

\fill (5,1) circle (2pt) ;
\fill (5,2) circle (2pt) ;
\fill (5,3) circle (2pt) ;

\fill (6,0) circle (2pt) ;
\fill (6,1) circle (2pt) ;
\fill (6,2) circle (2pt) ;
\fill (6,3) circle (2pt) ;

\draw (0,4) -- (1,4) -- (0,3) -- (0,1) -- (1,1) -- (1,3) -- (2,4)  -- (3,4) -- (2,3) -- (3,3) -- (2,2) -- (2,1) 
   -- (3,1) -- (3,2) -- (4,3) -- (5,3) -- (4,2) -- (4,1) -- (5,1)  -- (5,2) -- (6,3) -- (6,0) ;

\node at (-0.3, 1) {\tiny 1} ;  
\node at (-0.3, 2) {\tiny 8} ;  
\node at (-0.3, 4) {\tiny 22} ;  
\node at (6.3, 0) {\tiny 0} ;  
\node at (6.3, 1) {\tiny 7} ;  
\node at (6.3, 3) {\tiny 21} ;   
\node at (3.3, 4) {\tiny 25} ;

\fill (8,1) circle (2pt) ;
\fill (8,2) circle (2pt) ;
\fill (8,3) circle (2pt) ;

\fill (9,1) circle (2pt) ;
\fill (9,2) circle (2pt) ;
\fill (9,3) circle (2pt) ;

\fill (10,1) circle (2pt) ;
\fill (10,2) circle (2pt) ;
\fill (10,3) circle (2pt) ;

\fill (11,1) circle (2pt) ;
\fill (11,2) circle (2pt) ;
\fill (11,3) circle (2pt) ;

\fill (12,1) circle (2pt) ;
\fill (12,2) circle (2pt) ;
\fill (12,3) circle (2pt) ;

\fill (13,1) circle (2pt) ;
\fill (13,2) circle (2pt) ;
\fill (13,3) circle (2pt) ;

\fill (14,0) circle (2pt) ;
\fill (14,1) circle (2pt) ;
\fill (14,2) circle (2pt) ;
\fill (14,3) circle (2pt) ;

\fill (15,3)  circle (2pt) ;

\draw (8,3) -- (9,3) -- (8,2)   -- (8,1) -- (9,1) -- (9,2) -- (10,3)  
    --(11,3) -- (10,2) -- (10,1) -- (11,1) -- (11,2) -- (12,3) -- (13,3) -- (12,2) 
   -- (12,1) -- (13,1) -- (13,2) -- (14,3) -- (15,3) -- (14,2) -- (14,0)   ;

\node at (7.7, 1) {\tiny 1} ;  
\node at (7.7, 2) {\tiny 8} ;  
\node at (7.7, 3) {\tiny 15} ;  
\node at (14.3, 0) {\tiny 0} ;  
\node at (14.3, 1) {\tiny 7} ;  
\node at (15.3, 3) {\tiny 22} ;  

\end{tikzpicture}\end{center}
\end{figure}
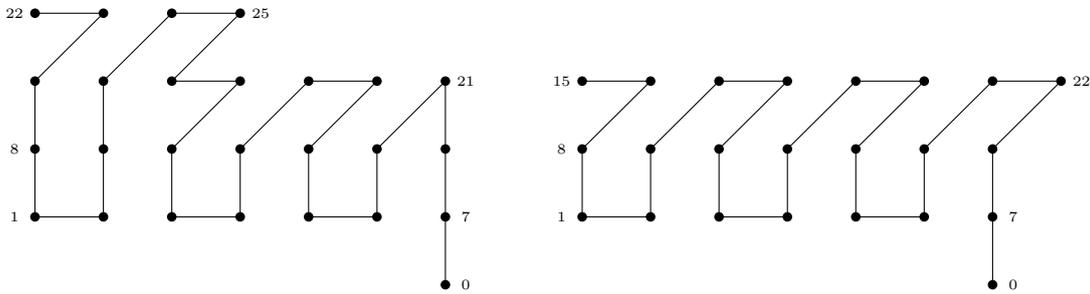

Case 4: $b+c \geq x$ with odd $x$ and odd~$r \neq 1$ or $r=0$.  When~$r$ is odd, start with the $\omega$-construction~$\mathbf{h_3}$; see the first diagram of Figure~\ref{fig:omega2} for a visual reminder.  We may use a corner flap and $\gamma$-moves at any (or all) of $v-x, v-x+2, \ldots, v-r$ and $v-r+3, v-r+5, \ldots, v-2$.  This is sufficient to replace up to~$x$ $x$-edges with~$(x+1)$-edges as required.  The first diagram of Figure~\ref{fig:42c} illustrates this with the maximum number of $\gamma$-moves and a corner flap for~$x=7$ and~$r=5$.  
The same construction works for~$r=0$, where the $\gamma$-moves can happen at~$v-x$  and at $v-x+3, v-x+5, \ldots, v-2$.  This is illustrated in the second diagram of Figure~\ref{fig:42c} with the maximum number of~$\gamma$-moves and a corner flap for~$x=7$.

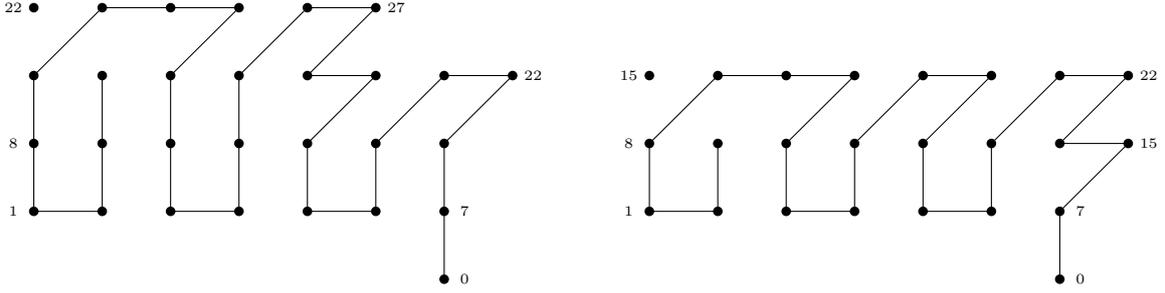
\begin{figure}[tp]
\caption{Standard linear realizations for~$\{1^8, 7^{12}, 8^{7}\}$ and~$\{1^8, 7^7, 8^{7}\}$.}\label{fig:42c}
\begin{center}
\begin{tikzpicture}[scale=0.9, every node/.style={transform shape}]

\fill (0,1) circle (2pt) ;
\fill (0,2) circle (2pt) ;
\fill (0,3) circle (2pt) ;
\fill (0,4) circle (2pt) ;

\fill (1,1) circle (2pt) ;
\fill (1,2) circle (2pt) ;
\fill (1,3) circle (2pt) ;
\fill (1,4) circle (2pt) ;

\fill (2,1) circle (2pt) ;
\fill (2,2) circle (2pt) ;
\fill (2,3) circle (2pt) ;
\fill (2,4) circle (2pt) ;

\fill (3,1) circle (2pt) ;
\fill (3,2) circle (2pt) ;
\fill (3,3) circle (2pt) ;
\fill (3,4) circle (2pt) ;

\fill (4,1) circle (2pt) ;
\fill (4,2) circle (2pt) ;
\fill (4,3) circle (2pt) ;
\fill (4,4) circle (2pt) ;

\fill (5,1) circle (2pt) ;
\fill (5,2) circle (2pt) ;
\fill (5,3) circle (2pt) ;
\fill (5,4) circle (2pt) ;

\fill (6,0) circle (2pt) ;
\fill (6,1) circle (2pt) ;
\fill (6,2) circle (2pt) ;
\fill (6,3) circle (2pt) ;

\fill (7,3) circle (2pt) ;

\draw (1,3) -- (1,1) -- (0,1) -- (0,3) -- (1,4) -- (3,4) -- (2,3)  -- (2,1) -- (3,1) -- (3,3) -- (4,4) -- (5,4) 
   -- (4,3) -- (5,3) -- (4,2) -- (4,1) -- (5,1) -- (5,2) -- (6,3)  -- (7,3) -- (6,2) -- (6,0) ;

\node at (-0.3, 1) {\tiny 1} ;  
\node at (-0.3, 2) {\tiny 8} ;  
\node at (-0.3, 4) {\tiny 22} ;  
\node at (6.3, 0) {\tiny 0} ;  
\node at (6.3, 1) {\tiny 7} ;  
\node at (7.3, 3) {\tiny 22} ;   
\node at (5.3, 4) {\tiny 27} ;  

\fill (9,1) circle (2pt) ;
\fill (9,2) circle (2pt) ;
\fill (9,3) circle (2pt) ;

\fill (10,1) circle (2pt) ;
\fill (10,2) circle (2pt) ;
\fill (10,3) circle (2pt) ;

\fill (11,1) circle (2pt) ;
\fill (11,2) circle (2pt) ;
\fill (11,3) circle (2pt) ;

\fill (12,1) circle (2pt) ;
\fill (12,2) circle (2pt) ;
\fill (12,3) circle (2pt) ;

\fill (13,1) circle (2pt) ;
\fill (13,2) circle (2pt) ;
\fill (13,3) circle (2pt) ;

\fill (14,1) circle (2pt) ;
\fill (14,2) circle (2pt) ;
\fill (14,3) circle (2pt) ;

\fill (15,0) circle (2pt) ;
\fill (15,1) circle (2pt) ;
\fill (15,2) circle (2pt) ;
\fill (15,3) circle (2pt) ;

\fill (16,2) circle (2pt) ;
\fill (16,3) circle (2pt) ;

\draw (10,2) -- (10,1) -- (9,1)   -- (9,2) -- (10,3) -- (12,3) -- (11,2)  
    --(11,1) -- (12,1) -- (12,2) -- (13,3) -- (14,3) -- (13,2) -- (13,1) -- (14,1) 
   -- (14,2) -- (15,3) -- (16,3) -- (15,2) -- (16,2) -- (15,1) -- (15,0)   ;

\node at (8.7, 1) {\tiny 1} ;  
\node at (8.7, 2) {\tiny 8} ;  
\node at (8.7, 3) {\tiny 15} ;  
\node at (15.3, 0) {\tiny 0} ;  
\node at (15.3, 1) {\tiny 7} ;  
\node at (16.3, 2) {\tiny 15} ;  
\node at (16.3, 3) {\tiny 22} ;  

\end{tikzpicture}\end{center}
\end{figure}

Case 5: $b+c < x$.  If~$r$ is even, we use~$\mathbf{h_1}$ (which we may regardless of the parity of~$x$, as~$q=0$) and the argument goes through as in Case~1 above with~$r$ even and~$x$ even.    If~$r$ is odd, we use~$\mathbf{h_2}$ (which we may regardless of the parity of~$x$, as~$q=0$) and the argument goes through as in Case~3 above with~$r$ odd and~$x$ even.  
\end{proof}



We can now prove the main result of the section.

\begin{thm}\label{th:1xx+1}
Let $L = \{ 1^a, x^b, (x+1)^c \}$.  There is a standard linear realization for~$L$ whenever $a \geq x+1$. 
\end{thm}

\begin{proof}
The only multisets left unaccounted for by Lemmas~\ref{lem:smallb} and~\ref{lem:smallc} are those with both~$b \geq x$ and $c \geq x+1$. In such a case we may write $$L = \{ x^{kx}, (x+1)^{k(x+1)} \} \cup \{ 1^a, x^{b'}, (x+1)^{c'} \}$$ for some~$k \geq 1$ and either~$b' < x$ or~$c' < x+1$ (or both).    

Build a perfect realization for $\{ x^{kx}, (x+1)^{k(x+1)} \}$ by concatenating~$k$ copies of the perfect realization  $\{ x^{x}, (x+1)^{(x+1)} \}$ obtainable from Lemma~\ref{lem:perf1}. Then, concatenate this with the standard realization of $\{ 1^a, x^{b'}, (x+1)^{c'} \}$ given by the appropriate choice from  Lemma~\ref{lem:smallb} or~\ref{lem:smallc} to give the required standard realization of~$L$.
\end{proof}

Theorem~\ref{th:1xx+1} has an immediate consequence for some multisets with supports of sizes~4 or~5.

\begin{cor}\label{cor:bigsupp}{\rm (Supports of sizes~4 and~5)} 
If~$L = \{ 1^a,  x^b, (x+1)^c ,  y^d \}$  then~$L$ has a linear realization whenever~$a \geq x+y+1$.   If~$L = \{ 1^a,  x^b, (x+1)^c ,  y^d, (y+1)^e \}$ then~$L$ has a linear realization whenever~$a \geq x+y+2$.   
\end{cor}

\begin{proof}
For the first claim, write~$L =  \{ 1^{x+1}, x^b, (x+1)^c \} \cup \{ 1^{a - (x+1)} , y^d  \}$, noting that $a-(x+1) \geq y$.  To obtain the realization, concatenate the standard realization for $\{ 1^{x+1}, x^b, (x+1)^c \}$ obtained from Theorem~\ref{th:1xx+1} with the standard realization for~$\{ 1^{a - (x+1)} , y^d  \}$ from Theorem~\ref{th:omega}.

For the second claim, write~$L =  \{ 1^{x+1}, x^b, (x+1)^c \} \cup \{ 1^{a - (x+1)} , y^d , (y+1)^e \}$, noting that $a-(x+1) \geq y+1$. 
  To obtain the realization, concatenate the standard realizations for $\{ 1^{x+1}, x^b, (x+1)^c \}$ and $ \{ 1^{a - (x+1)} , y^d , (y+1)^e \}$ obtained from Theorem~\ref{th:1xx+1}.
\end{proof}

There are few results in the literature for multisets with support of size~4 or~5.  There are some results for fixed supports in~\cite{Avila23,OPPS,OPPS2,PP14} and all known results for variable supports follow from one of the constuctions for realizations 
for the multisets of the form $\{ 1^a, x_1^{b_1}, \ldots, x_k^{b_k} : 1 < x_1 < \cdots < x_k \}$ given by these three cases:
\begin{itemize}
\item all $x_i$ even and~$a \geq x_k-1$~\cite{OPPS},
\item $x_i$ even for~$i < k$, $x_k=x_{k-1}+1$ and~$a \geq 3x_k-4$~\cite{OPPS},
\item $a \geq 3x_k - 5 + \sum_{i=1}^k x_i $~\cite{HR09,OPPS2}.
\end{itemize}
The results of Corollary~\ref{cor:bigsupp} therefore considerably extend what is known for supports of size~4 and~5. 

The standard realizations of Theorem~\ref{th:1xx+1} may also be concatenated with other standard realizations from the literature to give further results.  We do not attempt to catalogue them here.

\section{The case $U = \{1,x,y \}$ with $x$ even}\label{sec:evenx}

In this section the primary goal is to prove Theorem~\ref{th:x=2}, which gives a strong result for~$x=2$ that we use in the next section when considering the BHR~Conjecture multisets with support of the form~$\{1,2,y\}$.  Theorem~\ref{th:xeven} then shows one way that the ideas from the proof of Theorem~\ref{th:x=2} generalize to other even~$x$.  

First we need a lemma that gives a slight improvement on  Corollary~\ref{cor:omega_gen} when~$x$ is even and~$y$ is odd.

\begin{lem}\label{lem:omega_eo}{\rm (Improving Corollary~\ref{cor:omega_gen})}
Suppose~$L = \{1^a,x^b,y^c\}$ with~$1<x<y$ and~$x$ even.  Then~$L$ has a linear realization when~$a \geq x+y-2$.  
\end{lem}

\begin{proof}
The case~$y=x+1$  is covered by Theorem~\ref{th:1xx+1}, so assume~$y > x+1$.  
Let~$b = q'x + r'$ and $c = q''y + r''$ with $0 \leq r' < x$ and $0 \leq r'' < y$.  

If~$y$ is even or~$a \geq x+y-1$ then the result follows from Corollary~\ref{cor:omega_gen}, so assume that~$y$ is odd and~$a = x+y-2$.  Further, if~$r''$ is even, $q''=0$ or~$r''=1$ then Theorem~\ref{th:omega} gives a standard realizations for $\{1^{x-1}, x^b\}$ and $\{1^{y-1}, y^c\}$ which may be concatenated to give a linear realization for~$L$, so assume that~$r''$ is odd and not~1 and that~$q''>0$.

Suppose~$b=1$ (in fact this construction works in exactly the same way for all odd~$b$, but we do not require that for the proof).   Then also~$r'$ is odd.  The linear realization~$\mathbf{h_2}$ for~$\{1^{x-1}, x \}$ has end-labels~0 and~1.    The linear realization~$\mathbf{h_3}$ for~$\{1^y, y^c \}$ includes an edge between~$v-1$ and~$v$, and hence its complement is a linear realization of the same multiset that includes an edge between~0 and~1.  Applying Lemma~\ref{lem:hrjoin} we acquire a linear realization for 
$$\left( \{1^{x-1}, x \} \cup  \{1^y, y^c \}  \right) \setminus \{1 \}  = \{1^{x+y-2}, x, y^c \} =L. $$

Now suppose that~$b>1$.  The primary approach is to construct a standard linear realization for~$\{ 1^{y-2}, x^2, y^c \}$ and use this in conjunction with Theorem~\ref{th:omega} and Lemma~\ref{lem:concat} to produce the required linear realization.  However, we deal with two subcases that each require a slightly different construction first: $r'' = y-2$ and $x=2$.

If~$r'' = y-2$, then
$$\mathbf{h_5} =  \left( \apph{}{i=0}{y-x-2} \Psi_i \right)  \uplus^x \ \Psi_{y-2} \uplus \Psi_{y-1} \uplus^x  \left( \apph{}{i=y-3}{y-x-2} \Psi_i \right),  $$
where we start from~0, is a standard linear realization for~$\{1^{y-3}, x^2, y^c \}$.  This is illustrated in the first diagram of Figure~\ref{fig:eo1}.  As $\omega(x,b-2)=x-1$, Theorem~\ref{th:omega} and Lemmas~\ref{lem:concat} and~\ref{lem:ones} give a linear realization for~$\{ 1^a, x^b, y^c \}$ for all~$a \geq x+y-4$ in this subcase; more than sufficient for our needs here. 

If~$x=2$ and~$r'' \neq y-2$ then
$$\mathbf{h_6} =   \left( \apph{}{i=0}{y-3} \Psi_i \right) \uplus^2 \  \Psi_{y-1}^{(q^*)} \uplus \Psi_{y-2} \uplus \Psi_{y-1}^{(0,q^*-1)},   $$
where we start from~0, is a standard linear realization for~$\{1^{y-1}, x, y^c \}$.  This is illustrated in the second diagram of Figure~\ref{fig:eo1}. As $\omega(x,b-1)=x-1$, Theorem~\ref{th:omega} and Lemma~\ref{lem:concat} give a linear realization for~$\{ 1^{x+y-2}, x^b, y^c \}$ in this subcase.

\begin{figure}[tp]
\caption{The standard linear realizations $\mathbf{h_5}$ and~$\mathbf{h_6}$ for~$\{ 1^6, 4^2, 9^{25} \}$ and~$\{ 1^8, 2, 9^{21} \}$ respectively.}\label{fig:eo1}
\begin{center}
\begin{tikzpicture}[scale=0.82, every node/.style={transform shape}]

\fill (0,1) circle (2pt) ;
\fill (0,2) circle (2pt) ;
\fill (0,3) circle (2pt) ;
\fill (0,4) circle (2pt) ;

\fill (1,1) circle (2pt) ;
\fill (1,2) circle (2pt) ;
\fill (1,3) circle (2pt) ;
\fill (1,4) circle (2pt) ;

\fill (2,1) circle (2pt) ;
\fill (2,2) circle (2pt) ;
\fill (2,3) circle (2pt) ;
\fill (2,4) circle (2pt) ;

\fill (3,1) circle (2pt) ;
\fill (3,2) circle (2pt) ;
\fill (3,3) circle (2pt) ;
\fill (3,4) circle (2pt) ;

\fill (4,1) circle (2pt) ;
\fill (4,2) circle (2pt) ;
\fill (4,3) circle (2pt) ;
\fill (4,4) circle (2pt) ;

\fill (5,1) circle (2pt) ;
\fill (5,2) circle (2pt) ;
\fill (5,3) circle (2pt) ;
\fill (5,4) circle (2pt) ;

\fill (6,1) circle (2pt) ;
\fill (6,2) circle (2pt) ;
\fill (6,3) circle (2pt) ;
\fill (6,4) circle (2pt) ;

\fill (7,1) circle (2pt) ;
\fill (7,2) circle (2pt) ;
\fill (7,3) circle (2pt) ;

\fill (8,1) circle (2pt) ;
\fill (8,2) circle (2pt) ;
\fill (8,3) circle (2pt) ;

\draw (0,1) -- (0,4) -- (1,4) -- (1,1) -- (2,1) -- (2,4) -- (3,4) -- (3,1) ;
\draw (4,1) -- (4,4) -- (5,4) -- (5,1) -- (6,1) -- (6,4) ;
\draw (7,1) -- (7,3) -- (8,3) -- (8,1) ;

\draw  plot [smooth] coordinates {(3,1) (5,0.5) (7,1)};
\draw  plot [smooth] coordinates {(4,1) (6,0.5) (8,1)};

\node at (-0.3, 1) {\tiny 0} ;  
\node at (-0.3, 2) {\tiny 9} ;  
\node at (-0.3, 4) {\tiny 27} ;  
\node at (6.3, 4) {\tiny 33} ;  
\node at (8.3, 1) {\tiny 8} ;  
\node at (8.3, 3) {\tiny 26} ;  
\node at (2.7, 1) {\tiny 3} ;  
\node at (3.7, 1) {\tiny 4} ;  
\node at (7.3, 1) {\tiny 7} ;  

\fill (10,1) circle (2pt) ;
\fill (10,2) circle (2pt) ;
\fill (10,3) circle (2pt) ;
\fill (10,4) circle (2pt) ;

\fill (11,1) circle (2pt) ;
\fill (11,2) circle (2pt) ;
\fill (11,3) circle (2pt) ;
\fill (11,4) circle (2pt) ;

\fill (12,1) circle (2pt) ;
\fill (12,2) circle (2pt) ;
\fill (12,3) circle (2pt) ;
\fill (12,4) circle (2pt) ;

\fill (13,1) circle (2pt) ;
\fill (13,2) circle (2pt) ;
\fill (13,3) circle (2pt) ;
\fill (13,4) circle (2pt) ;

\fill (14,1) circle (2pt) ;
\fill (14,2) circle (2pt) ;
\fill (14,3) circle (2pt) ;

\fill (15,1) circle (2pt) ;
\fill (15,2) circle (2pt) ;
\fill (15,3) circle (2pt) ;

\fill (16,1) circle (2pt) ;
\fill (16,2) circle (2pt) ;
\fill (16,3) circle (2pt) ;

\fill (17,1) circle (2pt) ;
\fill (17,2) circle (2pt) ;
\fill (17,3) circle (2pt) ;

\fill (18,1) circle (2pt) ;
\fill (18,2) circle (2pt) ;
\fill (18,3) circle (2pt) ;

\draw (10,1) -- (10,4) -- (11,4) -- (11,1) -- (12,1) -- (12,4) -- (13,4) -- (13,1) 
 -- (14,1) -- (14,3) -- (15,3) -- (15,1) -- (16,1) -- (16,3) ;
\draw (18,3) -- (17,3) -- (17,1) -- (18,1) -- (18,2) ;

\draw  plot [smooth] coordinates {(16,3) (17,3.5) (18,3)};

\node at (9.7, 1) {\tiny 0} ;  
\node at (9.7, 2) {\tiny 9} ;  
\node at (9.7, 4) {\tiny 27} ;  
\node at (13.3, 4) {\tiny 30} ;  
\node at (18.3, 1) {\tiny 8} ;  
\node at (18.3, 2) {\tiny 17} ;  
\node at (18.3, 3) {\tiny 26} ;  
\node at (15.7, 3) {\tiny 24} ;  

\end{tikzpicture}
\end{center}
\end{figure}
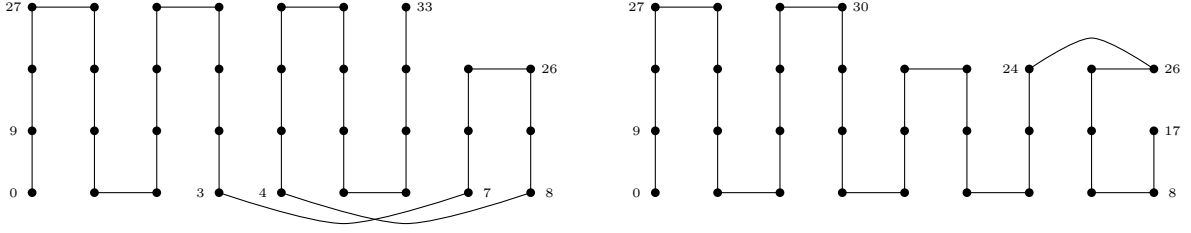

Finally, assume that $r'' \neq y-2$ and~$x \neq 2$.  Similarly to the $\omega$-construction~$\mathbf{h_4}$, the realization we use has two slightly different forms, depending on whether we need to use a single edge of one fauxset to successfully move between the fauxsets of differing sizes.  If~$r'' \neq y-4$ then let
$$\mathbf{h_7}  = \left( \apph{}{i=0}{y-x-2} \Psi_i \right) \uplus^x \ \Psi_{y-2} \uplus \Psi_{y-1} \uplus^x  \left( \apph{}{i=y-5}{y-x-2} \Psi_i \right) 
      \uplus \Psi_{y-4}^{(q^*)} \uplus \Psi_{y-3} \uplus \Psi_{y-2}^{(0,q^*-1)}  ;$$
 if $r'' = y-4$ then let
$$\mathbf{h_7} = \left( \apph{}{i=0}{y-x-2} \Psi_i \right) \uplus^x \ \Psi_{y-2} \uplus \Psi_{y-1} \uplus^x  \left( \apph{}{i=y-5}{y-x-2} \Psi_i \right) 
      \uplus \Psi_{y-4}^{(q^*-1,q^*)} \uplus \Psi_{y-3} \uplus \Psi_{y-2}^{(0,q^*-2)}  ,$$
where both versions start at~0.  Then~$\mathbf{h_7}$ is a standard linear realization of~$\{1^{y-2}, x^2, y^c \}$. 
The two variants of~$\mathbf{h_7}$ are illustrated in Figure~\ref{fig:eo2}.  As $\omega(x,b-2)=x-1$, Theorem~\ref{th:omega} and Lemmas~\ref{lem:concat} and~\ref{lem:ones} give a linear realization for~$\{ 1^a, x^b, y^c \}$ for all~$a \geq x+y-3$, which includes the required value of~$x+y-2$.
\end{proof}

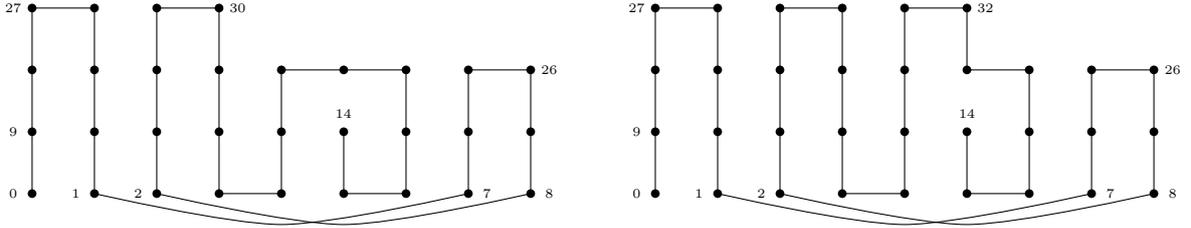
\begin{figure}[tp]
\caption{ The standard linear realization $\mathbf{h_7}$  for~$\{ 1^7, 6^2, 9^{21} \}$ and for~$\{ 1^7, 6^2, 9^{23} \}$.}\label{fig:eo2}

\begin{center}
\begin{tikzpicture}[scale=0.82, every node/.style={transform shape}]

\fill (0,1) circle (2pt) ;
\fill (0,2) circle (2pt) ;
\fill (0,3) circle (2pt) ;
\fill (0,4) circle (2pt) ;

\fill (1,1) circle (2pt) ;
\fill (1,2) circle (2pt) ;
\fill (1,3) circle (2pt) ;
\fill (1,4) circle (2pt) ;

\fill (2,1) circle (2pt) ;
\fill (2,2) circle (2pt) ;
\fill (2,3) circle (2pt) ;
\fill (2,4) circle (2pt) ;

\fill (3,1) circle (2pt) ;
\fill (3,2) circle (2pt) ;
\fill (3,3) circle (2pt) ;
\fill (3,4) circle (2pt) ;

\fill (4,1) circle (2pt) ;
\fill (4,2) circle (2pt) ;
\fill (4,3) circle (2pt) ;

\fill (5,1) circle (2pt) ;
\fill (5,2) circle (2pt) ;
\fill (5,3) circle (2pt) ;

\fill (6,1) circle (2pt) ;
\fill (6,2) circle (2pt) ;
\fill (6,3) circle (2pt) ;

\fill (7,1) circle (2pt) ;
\fill (7,2) circle (2pt) ;
\fill (7,3) circle (2pt) ;

\fill (8,1) circle (2pt) ;
\fill (8,2) circle (2pt) ;
\fill (8,3) circle (2pt) ;

\draw (0,1) -- (0,4) -- (1,4) -- (1,1) ;
\draw (2,1) -- (2,4) -- (3,4) -- (3,1) -- (4,1) -- (4,3) -- (6,3) 
  -- (6,1) -- (5,1) -- (5,2) ;
\draw (7,1) -- (7,3) -- (8,3) -- (8,1) ;

\draw  plot [smooth] coordinates {(1,1) (4,0.5) (7,1)};
\draw  plot [smooth] coordinates {(2,1) (5,0.5) (8,1)};

\node at (-0.3, 1) {\tiny 0} ;  
\node at (-0.3, 2) {\tiny 9} ;  
\node at (-0.3, 4) {\tiny 27} ;  
\node at (3.3, 4) {\tiny 30} ;  
\node at (8.3, 1) {\tiny 8} ;  
\node at (8.3, 3) {\tiny 26} ;  
\node at (0.7, 1) {\tiny 1} ;  
\node at (1.7, 1) {\tiny 2} ;  
\node at (7.3, 1) {\tiny 7} ;  
\node at (5, 2.3) {\tiny 14} ;  

\fill (10,1) circle (2pt) ;
\fill (10,2) circle (2pt) ;
\fill (10,3) circle (2pt) ;
\fill (10,4) circle (2pt) ;

\fill (11,1) circle (2pt) ;
\fill (11,2) circle (2pt) ;
\fill (11,3) circle (2pt) ;
\fill (11,4) circle (2pt) ;

\fill (12,1) circle (2pt) ;
\fill (12,2) circle (2pt) ;
\fill (12,3) circle (2pt) ;
\fill (12,4) circle (2pt) ;

\fill (13,1) circle (2pt) ;
\fill (13,2) circle (2pt) ;
\fill (13,3) circle (2pt) ;
\fill (13,4) circle (2pt) ;

\fill (14,1) circle (2pt) ;
\fill (14,2) circle (2pt) ;
\fill (14,3) circle (2pt) ;
\fill (14,4) circle (2pt) ;

\fill (15,1) circle (2pt) ;
\fill (15,2) circle (2pt) ;
\fill (15,3) circle (2pt) ;
\fill (15,4) circle (2pt) ;

\fill (16,1) circle (2pt) ;
\fill (16,2) circle (2pt) ;
\fill (16,3) circle (2pt) ;

\fill (17,1) circle (2pt) ;
\fill (17,2) circle (2pt) ;
\fill (17,3) circle (2pt) ;

\fill (18,1) circle (2pt) ;
\fill (18,2) circle (2pt) ;
\fill (18,3) circle (2pt) ;

\draw (10,1) -- (10,4) -- (11,4) -- (11,1) ;
\draw (12,1) -- (12,4) -- (13,4) -- (13,1) -- (14,1) -- (14,4) -- (15,4) 
  -- (15,3) -- (16,3) -- (16,1) -- (15,1) -- (15,2) ;
\draw (17,1) -- (17,3) -- (18,3) -- (18,1) ;

\draw  plot [smooth] coordinates {(11,1) (14,0.5) (17,1)};
\draw  plot [smooth] coordinates {(12,1) (15,0.5) (18,1)};

\node at (9.7, 1) {\tiny 0} ;  
\node at (9.7, 2) {\tiny 9} ;  
\node at (9.7, 4) {\tiny 27} ;  
\node at (15.3, 4) {\tiny 32} ;  
\node at (18.3, 1) {\tiny 8} ;  
\node at (18.3, 3) {\tiny 26} ;  
\node at (10.7, 1) {\tiny 1} ;  
\node at (11.7, 1) {\tiny 2} ;  
\node at (17.3, 1) {\tiny 7} ;  
\node at (15, 2.3) {\tiny 14} ;  

\end{tikzpicture}
\end{center}
\end{figure}

We can now prove the main result for~$x=2$.

\begin{thm}\label{th:x=2}{\rm (When $x=2$)}
Let $L = \{1^a, 2^b, y^c \}$, where $y \geq 5$ and $a+b \geq y-1$.  Then~$L$ has a linear realization when $a \geq 3$. 
\end{thm}

\begin{proof}
If~$a \geq y$ then the result follows from Lemma~\ref{lem:omega_eo}, so assume~$a < y$.  Let $c = q''y + r''$ with $0 \leq r'' < y$.

The main construction is via Lemma~\ref{lem:2rect}.  In order to use that result, write
$$\{ 1^a,2^b,y^c \} = \left( \{1, 2^{b_1} \} \cup \{1^a , 2^{b_2} , y^c \} \right) \setminus \{ 1 \}$$
where $b_2 = y-1 -a$ and~$b_1 = b -  b_2$.

The multiset $\{1, 2^{b_1} \}$ has a linear realization by Theorem~\ref{th:omega}; moreover the realization constructed in that proof has end-points~0 and~1.  

Suppose~$r'' > 1$.   Set~$v_1 = r''-1$ and~$v_2 = y - r'' +1$.
Choose~$a_1$ and~$a_2$, with~$a_1 + a_2 = a-1$, such that~$1 \leq a_1 \leq v_1 - 1$ and $1 \leq a_2 \leq v_2 - 1$, with the exception that~$a_1 = 0$ if~$r''=2$.  By Theorem~\ref{th:omega} and~Lemma~\ref{lem:ones} there is a standard realization for each of $\{ 1^{a_1}, 2^{r'' - 2 - a_1} \}$ and $\{ 1^{a_2} , 2^{y - r'' - a_2} \} $.  Applying Lemma~\ref{lem:2rect} and recalling that~$b_2 = y-1-a$ we obtain the linear realization for 
$$\{ 1^{a_1}, 2^{r'' - 2 - a_1} \} \cup \{ 1^{a_2} , 2^{y - r'' - a_2} \} \cup \{1 , y^c \} = \{1^a, 2^{b_2}, y^c \}$$
when $c \not\equiv 0,1 \pmod{y}$.  Applying~Lemma~\ref{lem:hrjoin} gives the result.

Now suppose~$r'' \in \{0,1 \}$.  We construct a linear realization using the same basic method as Lemma~\ref{lem:2rect} but requiring only ``one side" of the construction.  In the notation of that lemma, consider the following path, starting at the largest vertex of~$\Psi_0$:
$$\Psi_0 \ \uplus \  \left(  \apph{\mathbf{d}}{i=g_1+1}{g_{y-1}+1} \Psi_i  \right),$$
where~$\mathbf{d}$ is the sequence of differences given by a standard linear realization~$[ g_1, \ldots, g_{y-1}]$ for~$\{1^{a-1}, 2^{y-1-a} \}$.   This is a linear realization for~$\{1^a, 2^{y-1-a}, y^c \}$ and there is an edge between the vertices~0 and~1 where~$\Psi_0$ connects to~$\Psi_1$.  We again apply Lemma~\ref{lem:hrjoin} to obtain the result.
\end{proof}

Note that the construction in the proof of Theorem~\ref{th:x=2} is slightly stronger than required by the statement when~$c \equiv 0,1,2 \pmod{y}$.  In these cases we can relax the condition on~$a$ to~$a \geq 2$.

Similar ideas extend to arbitrary even~$x$.  This occurs straightforwardly when~$a+b+y$ is even; a new twist is needed when it is odd.  We require some realizations with a specified final edge length.   These are obtainable from previously presented methods; we collect these before moving on to the theorem.

\begin{lem}\label{lem:finaledge}{\rm (Final edge)} 
Let~$x \geq 2$ and~$L = \{ 1^a, x^b \}$.  For~$a=x-1$ and~$b \geq 2$, there is a standard realization of~$L$ whose final edge is an~$x$-edge.
For~$a = x$ and~$b \leq x$, there is a standard linear realization for~$L$ whose final edge is a $1$-edge.
\end{lem}

\begin{proof}  Let~$a= x-1$.
If $2 \leq b \leq x-1$ then the final edge of the linear realization~$\mathbf{h_2}$ of~$\{1^{x-1},x^b\}$ is an~$x$-edge (in the notation of that result, as $q'=0$ we can always hop from~$\varphi_r$ to~$\varphi_{r-1}$, so we are not concerned with the parity of~$x-r$).  If~$b > x-1$, then whichever linear realization from the proof of Theorem~\ref{th:omega}  of~$\{1^{x-1},x^b\}$ is successful---either~$\mathbf{h_1}$ or~$\mathbf{h_2}$ (possibly both)---has an~$x$-edge as a final edge. 

Now let~$a=x$ and~$b \leq x$.  When~$b=0$ we have the perfect realization~$[0,1,\ldots, x-1]$ whose final edge is a~1-edge (as are all of the others).   
When~$0 < b < x$ and~$b=1$ or~$b$ is even, then the last edge of the successful $\omega$-construction from Theorem~\ref{th:omega} for~$\{1^{x-1}, x^b\}$, again either~$\mathbf{h_1}$ or~$\mathbf{h_2}$, is a $1$-edge.  Obtain the required realization by concatenating adding an additional $1$-edge at the start; that is, concatenate it with the perfect realization~$[0,1]$.
For odd~$b$ in the range~$1 < b \leq x$, the last edge of the  standard realization~$\mathbf{h_3}$ for~$\{1^x, x^b\}$ is a~$1$-edge.

This leaves the case with even~$b$ and~$b=x$.  A construction very similar to~$\mathbf{h_3}$ works, with the difference that the tail curl moves through the penultimate coset at the lowest value rather than the highest:
$$\mathbf{h'_3} =  \Psi_0 \ \uplus \left( \apph{}{k=x-1}{3} \Psi_k \right) \uplus \Psi_2^{(0)} \uplus \Psi_1 \uplus \Psi_2^{(1)}. $$ 
(In fact, $\mathbf{h'_3}$ gives an alternative successful construction for all even~$b$ in the range $1 < b < x$.)  

\begin{figure}[tp]
\caption{The standard realization~$\mathbf{h'_3}$ for $\{ 1^8, 8^8 \}$.}\label{fig:h3dash}
\begin{center} \begin{tikzpicture}[scale=0.9, every node/.style={transform shape}]

\fill (0,1) circle (2pt) ;
\fill (0,2) circle (2pt) ;

\fill (1,1) circle (2pt) ;
\fill (1,2) circle (2pt) ;

\fill (2,1) circle (2pt) ;
\fill (2,2) circle (2pt) ;

\fill (3,1) circle (2pt) ;
\fill (3,2) circle (2pt) ;

\fill (4,1) circle (2pt) ;
\fill (4,2) circle (2pt) ;

\fill (5,1) circle (2pt) ;
\fill (5,2) circle (2pt) ;

\fill (6,1) circle (2pt) ;
\fill (6,2) circle (2pt) ;

\fill (7,0) circle (2pt) ;
\fill (7,1) circle (2pt) ;
\fill (7,2) circle (2pt) ; 

\draw (1,2) -- (0,2) -- (0,1) -- (1,1) -- (2,1) -- (2,2) -- (3,2) -- (3,1) -- (4,1) -- (4,2) -- (5,2) -- (5,1) -- (6,1) -- (6,2) -- (7,2) -- (7,0) ;

\node at (-0.3, 1) {\tiny 1} ;  
\node at (-0.3, 2) {\tiny 9} ;  
\node at (1, 2.3) {\tiny 10};
\node at (7.3, 0) {\tiny 0} ;    
\node at (7.3, 1) {\tiny 8} ;  
\node at (7.3, 2) {\tiny 16} ;  

\end{tikzpicture}\end{center}
\end{figure}
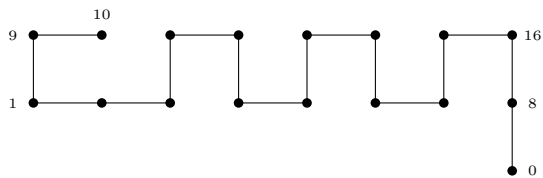

The realization~$\mathbf{h'_3}$ is illustrated for $\{1^8, 8^8\}$ in Figure~\ref{fig:h3dash}.
\end{proof}

\begin{thm}\label{th:xeven}{\rm (Even~$x$)} 
Let~$x$ be even, let~$y \geq 2x+1$ and let~$L = \{1^a, x^b, y^c \}$.  If~$a \geq 3x-2$ and $a+b \geq x + y - 1 $, then~$L$ has a linear realization.
\end{thm}

\begin{proof}
If~$a \geq x+y-2$ then the result follows from Lemma~\ref{lem:omega_eo}, so assume that~$a < x+y-2$. Also, if~$x=2$ then the result follows from Theorem~\ref{th:x=2}, so assume that~$x>2$.  Let~$c = q''y + r''$ with $0 \leq r'' < y$.  

As in the proof of Theorem~\ref{th:x=2}, the method is to use Lemma~\ref{lem:2rect}.  Here, we need consider two main cases depending on the parity of~$a+b+y$. 

First, assume that~$a+b+y$ is even, so $a+b$ and~$y$ have the same parity.     We write
$$\{ 1^a, x^b, y^c \}  = \left(  \{ 1^{x-1}, x^{b_1} \} \cup \{ 1^{a-x+2}, x^{b_2}, y^c \}   \right)  \setminus \{ 1 \}$$
where~$b_1$ is chosen to be odd and such that we have~$(a+x-2) + b_2 = y-1$.     This is possible as 
$(a+x-2) + b_2 = a + b + 1 - (x + b_1 - 1)$,
which has the opposite parity to~$a+b$ and hence the same parity as~$y-1$.

The multiset~$\{1^{x-1}, x^{b_1} \}$ has a linear realization by Theorem~\ref{th:omega} and as~$b_1$ is odd the realization constructed in that proof has end-points~0 and~1.  

Suppose~$r'' > 1$.  As in the proof of Theorem~\ref{th:x=2}, the method is to use intermediate standard realizations for $\{1^{a_1} , x^{r''-2-a_1} \}$ and for $\{1^{a_2}, x^{y-r'' -a_2} \}$, where in this more general situation we have~$a_1 + a_2 = a - x + 1$, and then apply Lemma~\ref{lem:2rect}.   As~$a \geq  3(x - 1)$ we have~$a-x+1 \geq 2(x-1)$ and so we can construct these in general using Theorem~\ref{th:omega} and Lemma~\ref{lem:ones}.  Note that if either $r'' -1 \leq x$ or~$y-r''+1 \leq x$ then one of the standard realizations must realize a multiset containing only 1-edges (and we often again get a looser condition on~$a$ in these situations). 

If $r'' \in \{0,1\}$ then the ``one-sided" version of Lemma~\ref{lem:2rect} works in the same way as it did for this case in the proof of Theorem~\ref{th:x=2}.

Now consider the case where~$a+b+y$ is odd, so $a+b$ and~$y$ have opposite parities.   The above method does not work as-is because it would require a linear realization with end-points~0 and~1 for~$\{1^{x-1}, x^{b_1} \}$ for some even~$b_1$.  Such realizations do not exist in general.  The solution is to construct a realization for either~$\{1^{a-1}, x^b, y^{c+1} \}$ or $\{1^a, x^{b-1}, y^{c+1} \}$ (which of these we choose depends on the exact parameters involved) in such a way that we may adjust it to become a linear realization for~$L$.  

Given the realization using~$c+1$ edges of length~$y$, we need to find a Hamiltonian path that uses one additional 1-edge or~$x$-edge and one fewer $y$-edge.  To gain a 1-edge we require one of the intermediate realizations with support~$\{1,x\}$ to have a~1-edge as the final edge; to gain an $x$-edge we require one of these intermediate realizations an~$x$-edge as the final edge.  We may then  use the tail curl trick to make the switch.  Figure~\ref{fig:1417} illustrates the process with~$U = \{1,4,17\}$ and~$c=45$.  The intermediate realization used on the left side of the picture realizes~$\{1^5\}$ and has a~1-edge as the final edge; the intermediate realization used on the right side of the picture realizes~$\{1^4, 4^6 \}$ and has a~4-edge as the final edge.

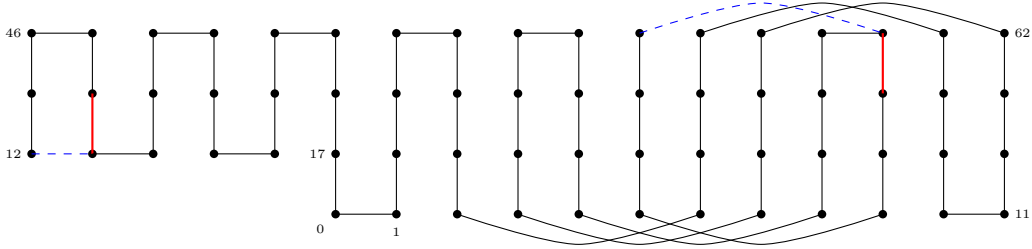
\begin{figure}[tp]
\caption{The solid lines are a linear realization for $\{1^{10}, 4^6, 17^{46}\}$.  Replacing the bolded red solid line on the left of the picture with the dashed blue line incident with it gives a linear realization for  $\{1^{11}, 4^6, 17^{45}\}$.  If we instead  replace the bolded red solid line on the right of the picture with the dashed blue line incident with it we obtain a linear realization for  $\{1^{10}, 4^7, 17^{45}\}$.   }\label{fig:1417}

\begin{center}
\begin{tikzpicture}[scale=0.8, every node/.style={transform shape}]

\fill (1,2) circle (2pt) ;
\fill (1,3) circle (2pt) ;
\fill (1,4) circle (2pt) ;

\fill (2,2) circle (2pt) ;
\fill (2,3) circle (2pt) ;
\fill (2,4) circle (2pt) ;

\fill (3,2) circle (2pt) ;
\fill (3,3) circle (2pt) ;
\fill (3,4) circle (2pt) ;

\fill (4,2) circle (2pt) ;
\fill (4,3) circle (2pt) ;
\fill (4,4) circle (2pt) ;

\fill (5,2) circle (2pt) ;
\fill (5,3) circle (2pt) ;
\fill (5,4) circle (2pt) ;

\fill (6,1) circle (2pt) ;
\fill (6,2) circle (2pt) ;
\fill (6,3) circle (2pt) ;
\fill (6,4) circle (2pt) ;

\fill (7,1) circle (2pt) ;
\fill (7,2) circle (2pt) ;
\fill (7,3) circle (2pt) ;
\fill (7,4) circle (2pt) ;

\fill (8,1) circle (2pt) ;
\fill (8,2) circle (2pt) ;
\fill (8,3) circle (2pt) ;
\fill (8,4) circle (2pt) ;

\fill (9,1) circle (2pt) ;
\fill (9,2) circle (2pt) ;
\fill (9,3) circle (2pt) ;
\fill (9,4) circle (2pt) ;

\fill (10,1) circle (2pt) ;
\fill (10,2) circle (2pt) ;
\fill (10,3) circle (2pt) ;
\fill (10,4) circle (2pt) ;

\fill (11,1) circle (2pt) ;
\fill (11,2) circle (2pt) ;
\fill (11,3) circle (2pt) ;
\fill (11,4) circle (2pt) ;

\fill (12,1) circle (2pt) ;
\fill (12,2) circle (2pt) ;
\fill (12,3) circle (2pt) ;
\fill (12,4) circle (2pt) ;

\fill (13,1) circle (2pt) ;
\fill (13,2) circle (2pt) ;
\fill (13,3) circle (2pt) ;
\fill (13,4) circle (2pt) ;

\fill (14,1) circle (2pt) ;
\fill (14,2) circle (2pt) ;
\fill (14,3) circle (2pt) ;
\fill (14,4) circle (2pt) ;

\fill (15,1) circle (2pt) ;
\fill (15,2) circle (2pt) ;
\fill (15,3) circle (2pt) ;
\fill (15,4) circle (2pt) ;

\fill (16,1) circle (2pt) ;
\fill (16,2) circle (2pt) ;
\fill (16,3) circle (2pt) ;
\fill (16,4) circle (2pt) ;

\fill (17,1) circle (2pt) ;
\fill (17,2) circle (2pt) ;
\fill (17,3) circle (2pt) ;
\fill (17,4) circle (2pt) ;

\draw (1,2) -- (1,4) -- (2,4) -- (2,3) ;
\draw[thick, red] (2,3) -- (2,2)  ;
\draw (2,2) -- (3,2) -- (3,4) -- (4,4) -- (4,2) -- (5,2) -- (5,4) -- (6,4) -- (6,1) -- (7,1) -- (7,4) -- (8,4) --(8,1) ;

\draw (9,1) -- (9,4) -- (10,4) -- (10,1) ;
\draw (11,1) -- (11,4) ;
\draw (12,1) -- (12,4) ;
\draw (13,1) -- (13,4) ;
\draw (14,1) -- (14,4)  -- (15,4);
\draw[thick, red] (15,4) -- (15,3) ;
\draw (15,3) -- (15,1) ;
\draw (16,4) -- (16,1) -- (17,1) -- (17,4)  ;

\draw  plot [smooth] coordinates {(8,1) (10,0.5) (12,1)};
\draw  plot [smooth] coordinates {(9,1) (11,0.5) (13,1)};
\draw  plot [smooth] coordinates {(10,1) (12,0.5) (14,1)};
\draw  plot [smooth] coordinates {(11,1) (13,0.5) (15,1)};

\draw  plot [smooth] coordinates {(12,4) (14,4.5) (16,4)};
\draw  plot [smooth] coordinates {(13,4) (15,4.5) (17,4)};

\draw[dashed, blue] (1,2) -- (2,2) ;
\draw[dashed, blue]  plot [smooth] coordinates {(11,4) (13,4.5) (15,4)};

\node at (0.7, 2) {\tiny 12} ;  
\node at (0.7, 4) {\tiny 46} ;  
\node at (5.75, 0.75) {\tiny 0} ;  
\node at (5.7, 2) {\tiny 17} ;  
\node at (7, 0.7) {\tiny 1} ;  
\node at (17.3, 1) {\tiny 11} ;  
\node at (17.3, 4) {\tiny 62} ;  

\end{tikzpicture}
\end{center}
\end{figure}

It now remains to show that we can apply this process in all cases with~$a+b+y$ odd. 
We construct the required realizations of $\{1^{a-1}, x^b, y^{c+1} \}$ or $\{1^a, x^{b-1}, y^{c+1} \}$ using the method of the first part of the proof.   

First, consider the case~$a < y-x-2$.  Use the method to construct a realization for~$\{1^a, x^{b-1}, y^{c+1} \}$.  In doing so, we find that the value of~``$b_2$" is~$y-x-a$, which is greater than 2.  Therefore at least one of the intermediate standard realizations has at least two $x$-edges and can be chosen to end with an~$x$-edge by Lemma~\ref{lem:finaledge}.  Applying the tail curl trick gives a realization for~$L$.

Next, consider the alternative case:~$y-x-2 \leq a < x+y-2$.   Use the method to construct a realization for~$\{1^{a-1}, x^b, y^{c+1} \}$.  This time we find the value of~``$b_2$" to be~$y-x-a + 1$, which is at most~3 and so less than~$x$. (Note that where the argument requires $a \geq 3(x-1)$ we now need $a-1 \geq  3(x-1)$, giving the constraint $a \geq 3x-2$ in the statement of the theorem.)  Using Lemma~\ref{lem:finaledge}, we may choose one of the intermediate standard realizations to end with a~1-edge.   Applying the tail curl trick gives a realization for~$L$.
\end{proof}

\section{Implications for the BHR Conjecture}\label{sec:bhr}

As noted in the introduction, the primary reason for interest in linear realizations is their usefulness in resolving instances of the BHR Conjecture.  In this section we use some modular arithmetic to show that the results obtained for linear realizations in the literature and here can be used to give results for the BHR Conjecture that are much more wide-ranging than currently known.

As described in the introduction, we may use automorphisms of~$\Z_v$ to move between equivalent instances of the BHR Conjecture in~$K_v$.  In particular, given a multiset~$\{1^a, x^b, y^c\}$ with~$x$ and~$y$ coprime to~$v$ we get two further multisets for which the BHR Conjecture is equivalent: $\{ \widehat{x^{-1}}^a   , 1^b , \widehat{x^{-1}y}^c  \}$ and  $\{ \widehat{y^{-1}}^a  ,  \widehat{xy^{-1}}^b , 1^c  \}$.  If at least one of these multisets has a linear realization then each has a not-necessarily-linear realization.

In this section we focus on cases where we can show {\em all} multisets with a given support satisfy the BHR Conjecture, or come close to such results.  In a future paper we shall use similar techniques, and further new constructions, to give partial results for arbitrary supports of size~3~\cite{AO+}.

The general approach is to take bounds on~$a$,~$b$ and~$c$ that each imply the existence of a linear realization and show that if none of the bounds are met then $a+b+c < v-1$, a contradiction.

\begin{exa}\label{ex:v61}
Let~$v = 61$ and let~$L = \{ 1^a, 4^b, 13^c \}$.  If any of~$a$, $b$ or~$c$ is~$0$, then~$L$ is realizable by Theorem~\ref{th:known}.1, so assume $a,b,c > 0$.  The equivalent multisets to~$L$ are~$L' = \{ 15^a, 1^b, 12^c \}$ and~$L'' = \{14^a, 5^b, 1^c \}$.  We have that~$L$ has a linear realization when~$a \geq 15$ by Lemma~\ref{lem:omega_eo}, that~$L'$ has a linear realization when~$b \geq 22$ also by Lemma~\ref{lem:omega_eo}, and that~$L''$ has a linear realization when~$c \geq 18$ by Corollary~\ref{cor:omega_gen}.  Hence~$L$ has a realization if any one of these three conditions hold.  If none of them hold then $$a+b+c < 15 + 25 + 18 = 58 < 60 = v-1,$$ a contradiction.  Therefore, when~$v=61$ all multisets with support~$\{1,4,13\}$ satisfy the BHR Conjecture.
\end{exa}

To obtain general results we need to understand more about how a support is related to its equivalent ones.  For example, in~$\Z_{61}$ the support $\{ 1, 8,9\}$ is equivalent to~$\{1, 23, 24\}$ and~$\{1,27,28\}$.  Notice that each is amenable to the application of Theorem~\ref{th:1xx+1}.  This situation applies more generally for supports of the form~$\{1,x,x+1\}$, which is proved and used in Theorem~\ref{th:bhr1xx+1}.

\begin{thm}\label{th:bhr1xx+1}{\rm (BHR Implications~1)} 
Let~$x > 1$.  For all $v \geq 2x^2 + 13x + 11$ with $\gcd(v,x) = \gcd(v,x+1) = 1$, the BHR~Conjecture holds for multisets with support~$\{1,x,x+1\}$. 
\end{thm}

\begin{proof}
First consider the supports that are equivalent to~$\{1,x,x+1\}$.  To find the first we multiply through by~$x^{-1}$ and reduce if necessary, to give~$\{ \widehat{x^{-1}} , 1, \widehat{1+x^{-1}} \}$.  Letting~$y= \min(\widehat{x^{-1}} ,\widehat{1+x^{-1}})$, this becomes~$\{1, y, y+1\}$.  To find the second  we multiply through by~$(x+1)^{-1}$ and reduce if necessary, to give 
$$\{ \widehat{(x+1)^{-1}} , \widehat{x(x+1)^{-1}} , 1 \} = \{ \widehat{(x+1)^{-1}} , \widehat{(x+1)^{-1} - 1} , 1 \} .$$  
Letting~$z= \min(\widehat{(x+1)^{-1}} , \widehat{(x+1)^{-1} - 1} )$, this becomes~$\{1, z, z+1\}$.  Hence all three supports have the structure that allows the application of Theorem~\ref{th:1xx+1}.

Suppose~$x$ is odd.  Then~$x \cdot \widehat{x^{-1}} = vk \pm 1$ for some~$k$ and so using the fact that~$\widehat{x^{-1}} < v/2$, we must have $\widehat{x^{-1}} \leq ((x-1)v/2x) + 1$.  Also, $x+1$ is even and at least~4, so $\widehat{(x+1)^{-1}} \leq (v-3)/2$. 
By Theorem~\ref{th:1xx+1}, the multiset~$\{1^a, x^b, (x+1)^c \}$ has a realization when~$a \geq x+1$, when $b \geq y+1$ or when $c \geq z+1$.   Suppose none of these inequalities are met.  Then
\begin{eqnarray*}
a+b+c &<& x+1 + y+1 + z+1 \\
  & \leq & x +1 \ + \  ((x-1)v/2x) + 3 \ + \ (v-3)/2 + 2 \\ 
  & = & (2x+9)/2 \ + \ v(2x-1)/2x. \\
\end{eqnarray*}
We see from the coefficient of~$v$ that for a fixed~$x$ and sufficiently large~$v$ we shall reach the desired contradiction~$a+b+c < v-1$.   A little more algebra shows that this occurs when~$v \geq 2x^2  + 11x$.

We repeat the process for~$x$ even.  In this case we have $\widehat{(x+1)^{-1}} \leq xv/(2x+2) + 1$ and~$\widehat{x^{-1}} \leq (v-3)/2$ (we may assume~$x>2$ as the BHR Conjecture holds for~$\{1,2,3\}$~\cite{CD10}).  We obtain the inequality
\begin{eqnarray*}
a+b+c &<& x+1 + y+1 + z+1 \\
 & \leq & x + 1 \  + \  (v-3)/2 + 2  \ + \  xv/(2x+2) + 3 \\
 & = & (2x+9)/2 \ + \ v(2x+1)/(2x+2) .\\
\end{eqnarray*}
The coefficient on~$v$ again indicates that sufficiently large~$v$ will lead to a contradiction; in this case we obtain $v \geq 2x^2 + 13x + 11$.
The result follows.
\end{proof}

Note that Theorem~\ref{th:bhr1xx+1} implies that, for any given~$x$, Buratti's Conjecture (that is, the BHR Conjecture with prime~$v$) holds for $\{1,x,x+1\}$ for all sufficiently large~$v$.   

The methodology of this section often works for values of~$v$ that do not satisfy the bound of Theorem~\ref{th:bhr1xx+1}.  Using it we may immediately resolve Buratti's Conjecture for three new supports.

\begin{cor}\label{cor:7910}{\rm (Computer findings)}
Let~$L$ be a multiset of size~$v-1$ with support~$\{1,7,8\}$, $\{1,9,10\}$ or $\{1,10,11\}$ such that~$v$ is coprime to each element of the support.  Then~$L$ has a realization.  In particular, Buratti's Conjecture holds for multisets with these supports.
\end{cor}

\begin{proof}
Let $L = \{ 1^a, x^b, (x+1)^c\}$ for $x \in \{7,9,10\}$ and suppose~$v$ is coprime to~$x$ and~$x+1$.  If~$v \geq 2x^2 + 13x +11$ then the result follows from Theorem~\ref{th:bhr1xx+1}.  Otherwise, when we consider~$L$ and its two equivalent multisets modulo~$v$, we find via a computer search that at least one of them is realizable via Theorem~\ref{th:1xx+1}.
\end{proof}

For~$x \leq 15$, Table~\ref{tab:badx} gives the only values of~$v$ with $\gcd(v,x) = \gcd(v,x+1) = 1$ for which the truth of the BHR Conjecture for support~$\{1, x, x+1\}$ does {\em not} follow from one of Theorems~\ref{th:known}.2,~\ref{th:known}.12 or~\ref{th:bhr1xx+1} or from appplying Theorem~\ref{th:1xx+1} as in the proof of Corollary~\ref{cor:7910}.

\begin{table}
\caption{Values of~$v$ and~$x$ with~$\gcd(v,x) = \gcd(v,x+1) = 1$ and~$x \leq 15$ for which the BHR Conjecture remains open for support~$\{1, x, x+1\}$.}\label{tab:badx}

$$
\begin{array}{r||c|c|c|c|c|c}
x & 8 & 11 & 12 & 13 & 14 & 15  \\
\hline 
v & 43 & 41,53,79 & 41,67,89 & 67 & 73 & 41, 43, 73, 103, 137, 167  \\
\end{array}
$$

\end{table}

To conclude, we slightly relax the focus for the section on instances where we can resolve all sufficiently large multisets with a given support to give a general result for multisets with support~$\{1,2,x\}$.  As with Theorem~\ref{th:bhr1xx+1}, and understanding of equivalent supports is important.  Here, however, this guides us on what to avoid.  For any odd~$v > 2x$, an equivalent support to~$\{1,2,x\}$ is~$\{1, y, (v-1)/2\}$ for some~$y$.  That~$(v-1)/2$ is as large as possible makes the results we have less effective, so we use a method that requires consideration of~$\{1,2,x\}$ and the other equivalent support only.

\begin{thm}\label{th:bhr12x}{\rm (BHR Implications~2)}
Let~$L = \{1^a, 2^b, x^c\}$.  If $v > 4x$ with $\gcd(v,x) = 1$ then the BHR Conjecture holds for~$L$, except possibly when~$x$ is odd and~$a \in \{1,2\}$.
\end{thm}

\begin{proof}
Assume that if~$x$ is odd then~$a \geq 3$.
By Theorems~\ref{th:known}.4 and~\ref{th:x=2}, the multiset~$L$ has a realization when~$a+b \geq x-1$, which is equivalent to~$c \leq v-x$.  Considering the equivalent multiset~$\{ \widehat{x^{-1}}^a, \widehat{2x^{-1}}^b, 1^c \}$, Corollary~\ref{cor:omega_gen} says that~$L$ has a realization when~$c \geq  \widehat{x^{-1}} +  \widehat{2x^{-1}}$.  Therefore,~$L$ has a realization whenever
$x + \widehat{x^{-1}} +  \widehat{2x^{-1}} \leq v$.

We divide the remainder of the proof into four cases, depending on the value of~$x^{-1}$.  First, suppose~$x^{-1}  < v/4$.  Then~$\widehat{2x^{-1}} = 2x^{-1} < v/2$ and 
$$ x + \widehat{x^{-1}} +  \widehat{2x^{-1}} = x + 3x^{-1} < v$$
as $x < v/4$.  

Second, suppose~$v/4 < x^{-1}  < v/2$.  Then~$v/2 < 2x^{-1} < v$, so~$\widehat{2x^{-1}} = v - 2x^{-1}$.  Hence
$$ x + \widehat{x^{-1}} +  \widehat{2x^{-1}} = v + x - x^{-1}  < v$$
as~$x < v/4 < x^{-1}$. 

Third, suppose~$v/2 < x^{-1}  < 3v/4$.  Then~$v < 2x^{-1} < 3v/2$, so $\widehat{2x^{-1}}= 2x^{-1} - v$ and $v/4 < \widehat{2x^{-1}} < v/2$.  Hence
$$ x + \widehat{x^{-1}} +  \widehat{2x^{-1}} =  x + x^{-1}  < v$$
as~$x < v/4$ and $x^{-1} < 3v/4$. 

Fourth, and finally, suppose $3v/4 < x^{-1}$.  Then $3v/2 < 2x^{-1}$ and~$\widehat{2x^{-1}}= 2v - 2x^{-1}$.  Hence
$$ x + \widehat{x^{-1}} +  \widehat{2x^{-1}} =  x + 3(v-x^{-1})  < v$$
as~$x < v/4$ and $x^{-1} > 3v/4$. 
\end{proof}

Theorems~\ref{th:bhr1xx+1} and~\ref{th:bhr12x} together imply Theorem~\ref{th:bhr}. Theorem~\ref{th:bhr1xx+1} implies the claim from the abstract and introduction that there are infinitely many sets~$U$ of size~3 for which there are infinitely many values of~$v$ where the BHR Conjecture holds for each multiset with support~$U$.

\end{document}